\documentclass[reqno,12pt]{amsart}

\usepackage{vmargin}
\setpapersize{A4}
       
\usepackage{amsmath} 
\usepackage{multirow}

\usepackage{amsfonts}   
\usepackage{color}
   
\usepackage{amssymb}

\usepackage{amscd}      
 
\usepackage{amsthm}

\usepackage{tikz,amsmath}
\usetikzlibrary{arrows}

\usepackage{amstext}      

\usepackage{graphicx}

\usepackage[utf8]{inputenc} 
\usepackage[T1]{fontenc} 

   \theoremstyle{plain}
   \newtheorem{thm}{Theorem}[section]
   \newtheorem{prop}[thm]{Proposition}
   \newtheorem{lemma}[thm]{Lemma}  
   \newtheorem{cor}[thm]{Corollary}
   \theoremstyle{defn}
   
    \newtheorem{defn}[thm]{Definition}
   \theoremstyle{remark}
   \newtheorem{obs}[thm]{Observation}
   \newtheorem{remark}[thm]{Remark}
   \newtheorem{example}[thm]{Example}

\usepackage{lipsum}

   \numberwithin{equation}{section}

\author{Johannes Christensen}

\email{johannes@math.au.dk}
\address{Institut for Matematik, Aarhus University, Ny Munkegade, 8000 Aarhus C, Denmark}

\title{Symmetries of the KMS simplex}

\begin{document}
\begin{titlepage} 
	\newcommand{\HRule}{\rule{\linewidth}{0.5mm}} 
	
	\center 
	
	
	\textsc{\LARGE Symmetries of the KMS simplex}\\[1.5cm] 
	
	\textsc{\Large Author: \\ 
Johannes Christensen}\\[0.5cm] 
	
	\textsc{\large Affiliation: \\
Institut for Matematik, Aarhus University}\\[0.5cm] 

\textsc{ Address: \\
Ny Munkegade, 8000 Aarhus C, Denmark}\\[0.5cm] 

\textsc{ Mail: \\
johannes@math.au.dk}\\[0.5cm] 

\textsc{ Telephone number: \\
+4587155723}\\[0.5cm]

	\HRule\\[1.5cm]

	{\large \bfseries Abstract}\\[0.4cm] 
	
A continuous groupoid homomorphism $c$ on a locally compact second countable Hausdorff \'etale groupoid $\mathcal{G}$ gives rise to a $C^{*}$-dynamical system in which every $\beta$-KMS state can be associated to a $e^{-\beta c}$-quasi-invariant measure $\mu$ on $\mathcal{G}^{(0)}$. Letting $\Delta_{\mu}$ denote the set of KMS states associated to such a $\mu$, we will prove that $\Delta_{\mu}$ is a simplex for a large class of groupoids, and we will show that there is an abelian group that acts transitively and freely on the extremal points of $\Delta_{\mu}$. This abelian group can be described using the support of $\mu$, so our theory can be used to obtain a description of all KMS states by describing the $e^{-\beta c}$-quasi-invariant measures. To illustrate this we will describe the KMS states for the Cuntz-Krieger algebras of all finite higher rank graphs without sources and a large class of continuous one-parameter groups.

	\HRule\\[1.5cm]
	
	\vfill\vfill\vfill 
	
	{\large\today} 
	
	
	 
	
	\vfill 
	
\end{titlepage}

\maketitle

\section{Introduction}
In recent years there has been a great deal of interest in describing KMS states for $C^{*}$-dynamical systems and many articles have been written about the subject. Often the $C^{*}$-dynamical systems investigated are given as a pair consisting of a groupoid $C^{*}$-algebra and a continuous one-parameter group arising from a continuous groupoid homomorphism. This is also the case for the articles about KMS states on $C^{*}$-algebras of higher rank graphs that have appeared the last several years, e.g. \cite{aHKR}, \cite{aHKR2}, \cite{aHLRS1} and \cite{aHLRS}. In \cite{aHLRS} the authors come to the conclusion that the simplex of KMS states for the $C^{*}$-dynamical systems they consider is "highly symmetric" in the sense that there is an abelian group that acts transitively and freely on the extremal points of the simplex. Inspired by this, the main purpose of this article is to investigate such symmetries using the groupoid picture of these $C^{*}$-algebras. We will do this by proving that the simplex of KMS states is symmetric for a large class of groupoid $C^{*}$-algebras and one-parameter groups given by continuous groupoid homomorphisms.

We will consider locally compact second countable Hausdorff \'etale groupoids $\mathcal{G}$ equipped with continuous homomorphisms $\Phi: \mathcal{G} \to A$ taking values in discrete abelian groups such that $\ker(\Phi) \cap \mathcal{G}_{x}^{x}=\{x\}$ for all $x \in \mathcal{G}^{(0)}$. Building on work of Renault, Neshveyev has described a bijection between the $\beta$-KMS states for one-parameter groups arising from a continuous groupoid homomorphism $c: \mathcal{G} \to \mathbb{R}$, and pairs consisting of a $e^{-\beta c}$-quasi-invariant probability measure $\mu$ on $\mathcal{G}^{(0)}$ and a specific kind of $\mu$-measurable field. Our main theorem describes how each $e^{-\beta c}$-quasi-invariant probability measure $\mu$ gives rise to a simplex $\Delta_{\mu}$ of KMS states associated to $\mu$, and how there for each $\mu$ is a subgroup $B$ of $A$ with the dual $\hat{B}$ of $B$ acting transitively and freely on the extremal points of $\Delta_{\mu}$. When there is only one $e^{-\beta c}$-quasi-invariant probability measure $\mu$ on $\mathcal{G}^{(0)}$ then $\Delta_{\mu}$ is the set of KMS states, and then our theorem implies that $\hat{B}$ acts transitively and freely on the extremal points of the simplex of KMS states.

The subgroup $B$ whose dual acts on the extreme points of $\Delta_{\mu}$ has a very concrete description involving the support of the measure $\mu$. This opens up the possibility of using these symmetries to describe the simplex of KMS states in cases where our methods so far have fallen short. The description we obtain has a precursor in Corollary 2.4 in \cite{N} and in the description in section 12 of \cite{aHLRS}, and with the theory developed it becomes possible to determine all the extremal KMS states by determining the $e^{-\beta c}$-quasi-invariant probability measures. We believe that this makes our main theorem a useful tool for giving concrete descriptions of KMS states on the groupoids under consideration. To illustrate this point, we will use the theorem to describe the KMS states for all Cuntz-Krieger algebras of finite higher rank $k$-graphs without sources and all continuous one-parameter groups obtained by taking an $r \in \mathbb{R}^{k}$ and mapping $\mathbb{R}$ into $\mathbb{T}^{k}$ by $t \to (e^{itr_{1}}, \dots, e^{itr_{k}})$ and composing with the gauge-action. This generalises Theorem 7.1 in \cite{aHLRS} where the KMS simplex for such actions is described for all strongly connected higher rank graphs.

\section{Notation and setting}
\subsection{$C^{*}$-dynamical systems.}
A $C^{*}$-dynamical system is a triple $(\mathcal{A}, \alpha, G)$ where $\mathcal{A}$ is a $C^{*}$-algebra, $G$ is a locally compact group and $\alpha$ is a strongly continuous representation of $G$ in $\text{Aut}(\mathcal{A})$. To ease notation we will denote the systems where $G=\mathbb{R}$ as $(\mathcal{A}, \alpha)$, in which case we call $\alpha=\{\alpha_{t}\}_{t \in \mathbb{R}}$ a continuous one-parameter group. For a $C^{*}$-dynamical system $(\mathcal{A}, \{\alpha_{t}\}_{t \in \mathbb{R}})$ and a $\beta \in \mathbb{R}$, a \emph{$\beta$-KMS state for $\alpha$} or a \emph{$\alpha$-KMS$_{\beta}$ state} is a state $\omega$ on $\mathcal{A}$ satisfying:
$$
\omega( x y) = \omega (y \alpha_{i \beta}(x)) 
$$
for all elements $x,y$ in a norm dense, $\alpha$-invariant $*$-algebra of entire analytic elements of $\alpha$, c.f. Definition 5.3.1 in \cite{BR}. The definition is independent of choice of norm dense, $\alpha$-invariant $*$-algebra of entire analytic elements. When there can be no confusion as to which $C^{*}$-dynamical system and $\beta \in \mathbb{R}$ we work with, we will denote the set of KMS states by $\Delta$. This is a simplex for unital $C^{*}$-algebras, and hence we can consider extremal KMS states, the set of which we will denote by $\partial \Delta$. In general when dealing with a compact and convex set $C$ in a locally convex topological vector space, we will use $\partial C$ to denote the extremal points of $C$.

\subsection{Groupoid $C^{*}$-algebras.}
Let $\mathcal{G}$ be a locally compact second countable Hausdorff \'etale groupoid with unit space $\mathcal{G}^{(0)}$ and range and source maps $r,s: \mathcal{G} \to \mathcal{G}^{(0)}$. Since $\mathcal{G}$ is \'etale $r$ and $s$ are local homeomorphisms, and we call an open set $W \subseteq \mathcal{G}$ a bisection when $r(W)$ and $s(W)$ are open and the maps $r|_{W}: W \to r(W)$ and $s|_{W}: W \to s(W)$ are homeomorphisms. For $x \in \mathcal{G}^{(0)}$ we set $\mathcal{G}_{x}:=s^{-1}(x)$ and $\mathcal{G}^{x}:=r^{-1}(x)$. The isotropy group at $x$ is then the set $\mathcal{G}_{x} \cap \mathcal{G}^{x}$ which we denote by $\mathcal{G}_{x}^{x}$. Let $C_{c}(\mathcal{G})$ denote the space of compactly supported continuous functions on $\mathcal{G}$. We can make this space into a $*$-algebra by defining a product:
$$
(f_{1}* f_{2}) (g)= \sum_{h \in \mathcal{G}^{r(g)}} f_{1}(h) f_{2}(h^{-1} g) \qquad \forall g \in \mathcal{G}
$$
and an involution by $f^{*}(g)=\overline{f(g^{-1})}$ for all $g \in \mathcal{G}$. When completing $C_{c}(\mathcal{G})$ in the full norm, see Definition 1.12 in chapter II of \cite{Re}, we obtain the full groupoid $C^{*}$-algebra $C^{*}(\mathcal{G})$. Since $\mathcal{G}$ is second countable it follows that $C^{*}(\mathcal{G})$ is separable. The full norm has the property that the map $C_{c}(\mathcal{G}) \to C_{c}(\mathcal{G}^{(0)})$ which restricts functions to $\mathcal{G}^{(0)}$ extends to a conditional expectation $P : C^{*}(\mathcal{G}) \to C_{0}(\mathcal{G}^{(0)})$.

Taking a continuous groupoid homomorphism $c:\mathcal{G}\to \mathbb{R}$, i.e. a continuous function $c: \mathcal{G} \to \mathbb{R}$ with $c(gh)=c(g)+c(h)$ when $s(g)=r(h)$, then for each $t\in \mathbb{R}$ we can define an automorphism $\alpha^{c}_{t}$ of $C_{c}(\mathcal{G})$ by setting:
\begin{equation*} 
\alpha_{t}^{c}(f)(g)=e^{it c(g)}f(g) \qquad \forall g \in \mathcal{G} .
\end{equation*}
The map $\alpha^{c}_{t}$ then extends to an automorphism of $C^{*}(\mathcal{G})$, and $\{ \alpha_{t}^{c}\}_{t \in \mathbb{R}}$ becomes a continuous one-parameter group. For the $C^{*}$-dynamical system $(C^{*}(\mathcal{G}), \{ \alpha_{t}^{c}\}_{t \in \mathbb{R}})$ the $*$-algebra $C_{c}(\mathcal{G})$ is norm-dense, $\alpha^{c}$-invariant and consists of entire analytic elements for $\alpha^{c}$, so it is sufficient to check the KMS condition on elements in $C_{c}(\mathcal{G})$. 

\subsection{Neshveyevs Theorem.}
In Theorem 1.3 in \cite{N} Neshveyev provides a useful description of KMS states which we will outline in the following. For a continuous groupoid homomorphism $c: \mathcal{G} \to \mathbb{R}$ on a locally compact second countable Hausdorff \'etale groupoid $\mathcal{G}$, we say that a finite Borel measure $\mu$ on $\mathcal{G}^{(0)}$ is \emph{$e^{-\beta c}$-quasi-invariant} for some $\beta \in \mathbb{R}\setminus\{0\}$, if for every open bisection $W$ of $\mathcal{G}$ we have:
$$
\mu(s(W))=\int_{r(W)} e^{\beta c\left(r_{W}^{-1}(x)\right)}\ d \mu (x)
$$
where $r_{W}^{-1}$ is the inverse of $r_{W}:W\to r(W)$. In the terminology used in \cite{N} these measures are called quasi-invariant with Radon-Nikodym cocycle $e^{-\beta c}$. We will need the following observation about these measures: If $\mu$ is an $e^{-\beta c}$-quasi-invariant measure on $\mathcal{G}^{(0)}$ and $E \subseteq \mathcal{G}^{(0)}$ is an invariant Borel set, i.e. $r(s^{-1}(E))=E=s(r^{-1}(E))$, then the Borel measure $\mu_{E}$ given by $\mu_{E}(\mathcal{B}):=\mu(E \cap \mathcal{B})$ is a $e^{-\beta c}$-quasi-invariant measure. For a proof of this we refer the reader to the proof of Lemma 2.2 in \cite{Th2}. 

Let $\mu$ be a $e^{-\beta c}$-quasi-invariant measure. We say that a collection $\{\varphi_{x}\}_{x \in \mathcal{G}^{(0)}}$ consisting for each $x \in \mathcal{G}^{(0)}$ of a state $\varphi_{x}$ on $C^{*}(\mathcal{G}_{x}^{x})$ is a $\mu$-measurable field if for each $f \in C_{c}(\mathcal{G})$ the function:
$$
\mathcal{G}^{(0)} \ni x \to \sum_{g \in \mathcal{G}_{x}^{x}} f(g) \varphi_{x}(u_{g})
$$
is $\mu$-measurable, where $u_{g}$, $g \in \mathcal{G}_{x}^{x}$, denotes the canonical unitary generators of $C^{*}(\mathcal{G}_{x}^{x})$. We do not distinguish between $\mu$-measurable fields which agree for $\mu$-a.e. $x \in \mathcal{G}^{(0)}$. For any $\beta \in \mathbb{R}\setminus \{0\}$ Neshveyevs Theorem establishes a bijection between the $\beta$-KMS states for $\alpha^{c}$ on $C^{*}(\mathcal{G})$ and the pairs $(\mu, \{\varphi_{x}\}_{x \in \mathcal{G}^{(0)}})$ consisting of a $e^{-\beta c}$-quasi-invariant Borel probability measure $\mu$ on $\mathcal{G}^{(0)}$ and a $\mu$-measurable field of states $\{\varphi_{x}\}_{x \in \mathcal{G}^{(0)}}$ satisfying:
\begin{equation}\label{htn}
\varphi_{x}(u_{g})=\varphi_{r(h)} (u_{hgh^{-1}}) \quad \text{ for } \mu \text{-a.e } x \in \mathcal{G}^{(0)} \text{ and all } g \in \mathcal{G}_{x}^{x} \text{ and }  h \in \mathcal{G}_{x} .
\end{equation}
The KMS state $\omega$ corresponding to $(\mu, \{\varphi_{x}\}_{x \in \mathcal{G}^{(0)}})$ satisfies:
$$
\omega(f)= \int_{\mathcal{G}^{(0)}} \sum_{g \in \mathcal{G}_{x}^{x}} f(g) \varphi_{x}(u_{g}) \ d\mu(x) \qquad \forall f \in C_{c}(\mathcal{G}) .
$$

\subsection{Duality of abelian groups.}
For any locally compact abelian group $A$ we let $\widehat{A}$ denote the dual of $A$, which is the set of continuous characters $\xi :A \to \mathbb{T}$. Setting $(\xi_{1}\xi_{2})(a)=\xi_{1}(a)\xi_{2}(a)$ and $\xi^{-1}(a)=\overline{\xi(a)}$ for $a \in A$ defines a composition and inversion on $\hat{A}$ making it a group with the constant function $1$ as the unit. Using the compact-open topology $\hat{A}$ becomes a locally compact abelian group. In this article we will consider abelian groups $A$ that are discrete and countable, and then the compact-open topology on $\hat{A}$ is the topology of pointwise convergence, and $\hat{A}$ is compact with this topology. For any locally compact abelian group $A$ we have the identification $\widehat{(\widehat{A})}\simeq A$, and for any closed subgroup $H$ of $A$, defining the \emph{annihilator} $H^{\perp}$ as:
$$
H^{\perp}=\{ \xi \in \widehat{A}\ | \ \xi(h)=1 \text{ for all } h \in H\}
$$
we also have that $(H^{\perp})^{\perp}=H$, c.f. Lemma 2.1.3 in \cite{Ru}. When there can be no confusion about which group $A$ we work with, we denote its unit by $e_{0}$.

\subsection{Groupoids admitting an abelian valued homomorphism}
We will throughout this paper consider the following groupoids. 
\begin{defn} \label{d2}
We say that a groupoid $\mathcal{G}$ admits an abelian valued homomorphism $\Phi: \mathcal{G} \to A$, if $\mathcal{G}$ is a locally compact second countable Hausdroff \'etale groupoid with compact unit space $\mathcal{G}^{(0)}$, $A$ is some countable discrete abelian group and $\Phi:\mathcal{G} \to A$ is a continuous homomorphism such that $\ker(\Phi) \cap \mathcal{G}_{x}^{x} = \{x\}$ for all $x \in \mathcal{G}^{(0)}$. 
\end{defn}

\begin{remark}
For the rest of this paper all groupoids $\mathcal{G}$ will satisfy Definition \ref{d2} for some discrete countable abelian group $A$ and continuous homomorphism $\Phi: \mathcal{G} \to A$.
\end{remark}

It follows from Proposition 5.1 in chapter II in \cite{Re} that for a groupoid $\mathcal{G}$ that satisfies Definition \ref{d2}, there is an automorphism $\Psi_{\xi} \in \text{Aut}(C^{*}(\mathcal{G}))$ for every $\xi \in \hat{A}$ satisfying:
\begin{equation}\label{e2}
\Psi_{\xi}(f)(g)=\xi(\Phi(g))f(g) \quad \forall g\in \mathcal{G}
\end{equation}
whenever $f \in C_{c}(\mathcal{G} )$. Letting $\Psi: \xi\to \Psi_{\xi}$ then $(C^{*}(\mathcal{G}), \Psi , \hat{A})$ is a $C^{*}$-dynamical system. When there can be no doubt about which group $A$ we consider, we will often denote this action as the \emph{gauge-action}.

\begin{example}\label{ex22}
Let $(\Lambda , d)$ be a compactly aligned topological $k$-graph for some $k \in \mathbb{N}$, see e.g. \cite{Yeend}. Using $(\Lambda ,d)$ one can define a space of paths $X_{\Lambda}$ and for each $m \in \mathbb{N}^{k}$ a map $\sigma^{m}$ on $\{x\in X_{\Lambda} \ | \ d(x)\geq m\}$ and thereby obtain a groupoid:
\begin{align*}
G_{\Lambda}=\{(x,m,y)\in X_{\Lambda} \times \mathbb{Z}^{k} \times X_{\Lambda} \ | \ &  \exists p,q \in \mathbb{N}^{k} \text{ with } p \leq d(x),  \\
&q\leq d(y), \  p-q=m \text{ and } \sigma^{p}(x)=\sigma^{q}(y)  \}
\end{align*}
with composition $(x,m,y)(y,n,z)=(x,m+n,z)$, c.f. Definition 3.4 in \cite{Yeend}. Using Proposition 3.6 and Theorem 3.16 in \cite{Yeend} we can equip $G_{\Lambda}$ with a topology such that the homomorphism $G_{\Lambda}\ni (x,m,y) \to m \in \mathbb{Z}^{k}$ becomes continuous and $G_{\Lambda}$ satisfies Definition \ref{d2} when $X_{\Lambda}$ is compact. So the groupoid for the Toeplitz algebra of a compactly aligned topological $k$-graph with compact unit space satisfies Definition \ref{d2}. Since the groupoid for the Cuntz-Krieger algebra for a compactly aligned topological $k$-graph is a reduction of $G_{\Lambda}$, it also satisfies Definiton \ref{d2} when it has a compact unit space. This provides us with a lot of examples, see e.g. the ones listed in Example 7.1 in \cite{Yeend} where the unit space is compact. Most importantly for the content of this article, it implies that all groupoids of Cuntz-Krieger algebras of finite higher rank graphs without sources satisfy the criterion.
\end{example}

\begin{example}
Let $X$ be a compact second countable Hausdorff space and $A$ a countable abelian group, and denote by $\text{End}(X)$ the semigroup of surjective local homeomorphisms from $X$ to $X$. Let $P$ be a subsemigroup of $A$ cointaining the unit $e_{0}$ of $A$ with $PP^{-1}=P^{-1}P=A$, and let $\theta$ be a right action of $P$ on $X$ in the sense that $\theta:P \to \text{End}(X)$ satisfies $\theta_{e_{0}}=\text{id}_{X}$ and $\theta_{nm}=\theta_{m}\theta_{n}(=\theta_{n}\theta_{m})$ for all $n, m \in P$. Proposition 3.1 in \cite{ER} then informs us that:
$$
\mathcal{G}=\left\{ (x,g,y) \in X \times A \times X \ | \ \exists n,m \in P, \ g=nm^{-1}, \ \theta_{n}(x)=\theta_{m}(y)  \right\}
$$
is a groupoid with composition $(x,a,y)(y,b,z)=(x,ab,z)$. In Proposition 3.2 in \cite{ER} the authors define a topology that makes $\mathcal{G}$ a locally compact \'etale groupoid which is second countable and Hausdorff since $X$ is. The topology furthermore makes the homomorphism $\mathcal{G} \ni (x,a,y) \to a \in A$ continuous, so since $\mathcal{G}^{(0)} \simeq X$ is compact $\mathcal{G}$ satisfies Definition \ref{d2}. 
\end{example}

\section{The gauge-action and KMS states}
For this section we fix a groupoid $\mathcal{G}$ with an abelian valued homomorphism $\Phi: \mathcal{G} \to A$ as in Definition \ref{d2}. To prove our main theorem about symmetries in the KMS simplex, we first need a tool to control the size of the simplex. The purpose of this section is to show how there is an interplay between the gauge-action and the KMS states, and then use this interplay to gain some control over the size of the set of extremal KMS states

\begin{lemma} \label{l320}
Let $c :\mathcal{G} \to \mathbb{R}$ be a continuous groupoid homomorphism and $\beta \in \mathbb{R}$. Assume that $\omega$ is a $\alpha^{c}$-KMS$_{\beta}$ state on $C^{*}(\mathcal{G})$ satisfying $\omega \circ \Psi_{\xi}=\omega$ for all $\xi \in \hat{A}$. Then $\omega=\omega \circ P$.
\end{lemma}

\begin{proof}
Since $\Phi^{-1}(\{a\})$ is open for each $a \in A$ we can partition $\mathcal{G}$ into the open sets $\Phi^{-1}(\{a\})$, $a \in A \setminus \{e_{0}\}$, $\mathcal{G}^{(0)}$ and $\Phi^{-1}(\{e_{0}\})\setminus \mathcal{G}^{(0)}$. Using a partition of unity and linearity and continuity of $\omega$ and $\omega \circ P$ it follows that it is enough to prove that $\omega(f)=\omega(P(f))$ for $f \in C_{c}(\mathcal{G})$ supported in any of the above three kinds of sets. Suppose first that $\text{supp}(f) \subseteq \Phi^{-1}(\{a\})$ for a $a \neq e_{0}$. It follows that $P(f)=0$. Let $m$ denote the normalised Haar-measure on $\hat{A}$, the invariance of $\omega$ under $\Psi$ implies that:
$$
\omega(f)=\int_{\hat{A}} \omega(\Psi_{\xi}(f)) \ dm(\xi)= \int_{\hat{A}} \omega(f) \xi(a) \ dm(\xi) =\omega(f) \int_{\hat{A}}  \xi(a) \ dm(\xi) =0 .
$$
If $\text{supp}(f) \subseteq \mathcal{G}^{(0)}$ then $P(f)=f$ and $\omega(f)=\omega(P(f))$. For the last case notice that if $g \in\Phi^{-1}(\{e_{0}\})\setminus \mathcal{G}^{(0)}$ then $r(g)=s(g)$ would imply that $g \in \ker(\Phi) \cap \mathcal{G}_{r(g)}^{r(g)}$, contradicting that $g \notin \mathcal{G}^{(0)}$. Since $\mathcal{G}$ is \'etale it follows by linearity that we can assume $\text{supp}(f)$ is contained in an open set $U$ with $\overline{r(U)}\cap \overline{s(U)} = \emptyset$. Since $\mathcal{G}^{(0)}$ is compact we can pick $h \in C_{c}(\mathcal{G}^{(0)})$ with $h=1$ on $\overline{r(U)}$ and $\text{supp}(h) \subseteq \overline {s(U)}^{C}$. It follows using the definition of the product in $C_{c}(\mathcal{G})$ that $f=hf$ and $fh=0$, so using that $\omega$ is a $\alpha^{c}$-KMS$_{\beta}$  state and $h$ is fixed by $\alpha^{c}$ we get:
$$
\omega(f)=\omega(hf)=\omega(fh)=0
$$
which proves the Lemma.
\end{proof}

We can now use the gauge-action to control the size of the set of extremal KMS states.

\begin{thm}\label{opt}
Let $c: \mathcal{G} \to \mathbb{R}$ be a continuous groupoid homomorphism, $\beta \in \mathbb{R}$ and $\omega$ be an extremal $\beta$-KMS state for $\alpha^{c}$ on $C^{*}(\mathcal{G})$. Then for any extremal $\alpha^{c}$-KMS$_{\beta}$ state $\psi$ satisfying that $\psi \circ P=\omega \circ P$ there is a $\xi \in \hat{A}$ with $\psi=\omega \circ \Psi_{\xi}$.
\end{thm}
\begin{proof}
First we will argue that if some $\psi$ is a $\beta$-KMS state for $\alpha^{c}$ then $\psi \circ \Psi_{\xi}$ is also a $\beta$-KMS state for $\alpha^{c}$ for all $\xi \in \hat{A}$.
Equation \eqref{e2} implies that $\Psi_{\xi}(C_{c}(\mathcal{G})) \subseteq C_{c}(\mathcal{G})$ and that $\alpha^{c}_{t} \circ \Psi_{\xi} =  \Psi_{\xi} \circ \alpha^{c}_{t}$ for any $t\in \mathbb{R}$ and $\xi \in \hat{A}$. So for $f,g\in C_{c}(\mathcal{G})$ we get:
\begin{align*}
&\psi \circ \Psi_{\xi}(fg) = \psi (\Psi_{\xi}(f)\Psi_{\xi}(g)) = \psi (\Psi_{\xi}(g) \alpha_{i\beta}^{c} (\Psi_{\xi}(f))) \\
&= \psi (\Psi_{\xi}(g) \Psi_{\xi}( \alpha_{i\beta}^{c} (f))) = \psi \circ \Psi_{\xi}(g  \alpha_{i\beta}^{c} (f))
\end{align*}
and since $\psi \circ \Psi_{\xi}$ is clearly a state it is a $\beta$-KMS state for $\alpha^{c}$. That $\psi \circ \Psi_{\xi}$ is an extremal $\alpha^{c}$-KMS$_{\beta}$ state for any extremal $\beta$-KMS state $\psi$ and $\xi \in \hat{A}$ is straightforward to check using that $\Psi_{\xi}$ has inverse $\Psi_{\xi^{-1}}$. Now assume for contradiction that there is an extremal $\beta$-KMS state for $\alpha^{c}$, say $\psi$, with $\psi \circ P=\omega \circ P$ which is not on the form $\omega \circ \Psi_{\xi}$ for any $\xi \in \hat{A}$. It follows first that:
$$
\{\psi \circ \Psi_{\xi} \ | \ \xi \in \hat{A} \}
$$
is a set of extremal $\beta$-KMS states for $\alpha^{c}$, and then that:
$$
\{\psi \circ \Psi_{\xi} \ | \ \xi \in \hat{A} \} \cap \{\omega \circ \Psi_{\eta} \ | \ \eta \in \hat{A}\} = \emptyset
$$
since if $\psi \circ \Psi_{\xi} = \omega \circ  \Psi_{\eta}$ for some $\xi, \eta \in \hat{A}$, then $\psi = \omega \circ \Psi_{\eta \xi^{-1}}$, contradicting our choice of $\psi$. Denoting the $\beta$-KMS states for $\alpha^{c}$ by $\Delta$, we can define two functions from $\hat{A}$ to $ \Delta$:
$$
F_{1}(\xi)=\omega\circ \Psi_{\xi} \qquad F_{2}(\xi)=\psi \circ \Psi_{\xi} .
$$
Since $\Psi$ is strongly continuous, $F_{1}$ and $F_{2}$ are continuous when $\Delta$ has the weak$^{*}$-topology, so since $\hat{A}$ is compact $F_{1}(\hat{A}) \subseteq \partial \Delta$ and $F_{2}(\hat{A}) \subseteq \partial \Delta$ are two disjoint compact sets. Define two measures:
$$
\nu_{1}=m \circ F_{1}^{-1} \qquad , \qquad \nu_{2}=m \circ F_{2}^{-1}
$$
where $m$ is the normalised Haar-measure on $\hat{A}$, then $\nu_{1}$ and $\nu_{2}$ become Borel probability measures on $ \Delta$ supported on disjoint sets, and hence $\nu_{1} \neq \nu_{2}$. Since $\Delta$ is metrizable Choquet theory informs us, c.f. Theorem 4.1.11 in \cite{BR}, that  since $\nu_{1}(\partial \Delta)=1=\nu_{2}(\partial \Delta)$ both measures are maximal. So since $\Delta$ is a simplex they have two different barycenters $\omega_{1}\neq \omega_{2} \in \Delta$. For all $x \in C^{*}(\mathcal{G})_{+}$:
$$
\omega_{1}(x)=\int_{ \Delta} ev_{x}(\gamma) d\nu_{1}(\gamma) = \int_{\hat{A}} ev_{x}(\omega \circ \Psi_{\xi}) d m(\xi)=\int_{\hat{A}} \omega \circ \Psi_{\xi}(x) d m(\xi) .
$$
Notice that setting $\omega'(y):=\int_{\hat{A}} \omega \circ \Psi_{\xi}(y) d m(\xi)$ for $y \in C^{*}(\mathcal{G})$ defines a $\alpha^{c}$-KMS$_{\beta}$ state that is invariant under $\Psi$, and hence $\omega'(y)=\omega'(P(y))$. However $\Psi$ fixes $C(\mathcal{G}^{(0)})$ pointwise and hence $\omega'(y)=\omega'(P(y))=\omega (P(y))$. So $\omega_{1}(x)=\omega\circ P(x)$, and likewise $\omega_{2} (x)=\psi \circ P(x)$, contradicting that $\omega \circ P=\psi \circ P$ but $\omega_{1} \neq \omega_{2}$.
\end{proof}

\section{Extremal KMS states}

In this section we again let $\mathcal{G}$ be a groupoid with an abelian valued homomorphism $\Phi:\mathcal{G} \to A$ as in Definition \ref{d2}. To use Theorem \ref{opt} we need to obtain some extremal KMS state. The purpose of this section is to use Neshveyevs Theorem to obtain one extremal KMS state, and then use Theorem \ref{opt} to obtain the rest. To ease notation we will identify regular finite Borel measures on $\mathcal{G}^{(0)}$ with positive continuous linear functionals on $C(\mathcal{G}^{(0)})$.

\begin{lemma} \label{l24}
Fix a continuous groupoid homomorphism $c:\mathcal{G} \to \mathbb{R}$ and a $\beta \in \mathbb{R}\setminus \{0\}$. Let $\tilde{\Delta}$ be the set of $e^{-\beta c}$-quasi-invariant probability measures on $\mathcal{G}^{(0)}$, and for any $\mu \in \tilde{\Delta}$ let $\Delta_{\mu}$ be the set of $\alpha^{c}$-KMS$_{\beta}$ states $\omega$ on $C^{*}(\mathcal{G})$ with $\omega |_{C(\mathcal{G}^{(0)})}=\mu$. Then:
\begin{enumerate}
\item $\tilde{\Delta}$ is a compact convex set.
\item $\Delta_{\mu}$ is a compact convex set for any $\mu \in \tilde{\Delta}$.
\item A $\beta$-KMS state $\omega$ for $\alpha^{c}$ is extremal in the simplex of $\alpha^{c}$-KMS$_{\beta}$ states $\Delta$ if and only if $\mu:=\omega |_{C(\mathcal{G}^{(0)})} \in \partial \tilde{\Delta}$ and $\omega \in \partial \Delta_{\mu}$.
\end{enumerate}
\end{lemma}

\begin{proof}
That $\tilde{\Delta}$ is convex is straightforward to see. To see that it is closed, let $\{\mu_{n}\}_{n\in \mathbb{N}} \subseteq \tilde{\Delta}$ be a sequence such that $\mu_{n} \to \mu$ in the weak$^{*}$ topology. Then $\omega_{n}(x) :=\int_{\mathcal{G}^{(0)}} P(x)d\mu_{n}$ defines a sequence of $\beta$-KMS states that converges in the weak$^{*}$ topology, i.e. $\omega_{n} \to \omega$ for some $\beta$-KMS state $\omega$. It then follows that $\omega |_{C(\mathcal{G}^{(0)})} = \mu$, so that $\mu \in \tilde{\Delta}$. We leave the verification of $(2)$ to the reader.

For $(3)$, assume that $\omega \in \partial \Delta$ and let $\mu=\omega|_{C(\mathcal{G}^{(0)})}$. By Theorem 1.3 in \cite{N}, $\omega$ is given by a pair $(\mu, \{\varphi_{x}\}_{x \in \mathcal{G}^{(0)}})$ where $\{\varphi_{x}\}_{x \in \mathcal{G}^{(0)}}$ is a $\mu$-measurable field of states satisfying \eqref{htn}. Assume $\mu=\lambda \mu_{1}+(1-\lambda)\mu_{2}$ for some $\mu_{1}, \mu_{2} \in \tilde{\Delta}$ and $\lambda \in ]0,1[$. Then $\{\varphi_{x}\}_{x \in \mathcal{G}^{(0)}}$ is also $\mu_{1}$- and $\mu_{2}$-measurable and satisfies \eqref{htn} for $\mu_{1}$ and $\mu_{2}$, and then $(\mu_{1}, \{\varphi_{x}\}_{x \in \mathcal{G}^{(0)}})$ and $(\mu_{2}, \{\varphi_{x}\}_{x \in \mathcal{G}^{(0)}})$ represent two KMS states $\omega_{1}$ and $\omega_{2}$, satisfying:
\begin{equation*}
\omega(f)=\int_{\mathcal{G}^{(0)}} \sum_{g \in \mathcal{G}_{x}^{x}} f(g) \varphi_{x}(u_{g}) \ d\mu(x) 
=\lambda \omega_{1}(f) +(1-\lambda) \omega_{2}(f)
\end{equation*}
for all $f \in C_{c}(\mathcal{G})$. Since $\omega \in \partial \Delta$ this implies $\omega=\omega_{1}=\omega_{2}$, and hence $\mu=\mu_{1}=\mu_{2}$, proving that $\mu \in \partial \tilde{\Delta}$, and since $\Delta_{\mu}$ is contained in $\Delta$, we get that $\omega \in \partial \Delta_{\mu}$. The other implication in $(3)$ is straightforward.
\end{proof}

Using Lemma \ref{l24} we can now find an extremal KMS state for $(C^{*}(\mathcal{G}), \alpha^{c})$.

\begin{prop} \label{t22}
Let $c:\mathcal{G} \to \mathbb{R}$ be a continuous groupoid homomorphism, $\beta \in \mathbb{R} \setminus \{0\}$ and assume that $\mu$ is a $e^{-\beta c}$-quasi-invariant probability measure. Then for any $\xi \in \hat{A}$ there is a $\beta$-KMS state $\omega_{\xi}$ for $\alpha^{c}$ given by:
$$
\omega_{\xi}(f) = \int_{\mathcal{G}^{(0)}} \sum_{g \in \mathcal{G}_{x}^{x}} f(g) \xi(\Phi(g)) \ d\mu(x)
$$ 
for all $f \in C_{c}(\mathcal{G})$. For the function $1\in \hat{A}$ the state $\omega_{1}$ is an extremal point in $\Delta_{\mu}$.
\end{prop}

\begin{proof}
To prove the first claim it is enough to find a $\mu$-measurable field of states $\{\psi_{x}\}_{x \in \mathcal{G}^{(0)}}$ satisfying \eqref{htn} such that $\psi_{x}(u_{g})=\xi(\Phi(g))$ for all $g \in \mathcal{G}_{x}^{x}$ and all $x \in \mathcal{G}^{(0)}$. To do this fix a $\xi \in \hat{A}$ and define a $*$-homomorphism $H_{\xi}: C^{*}(A) \to \mathbb{C}$ by specifying that $H_{\xi}(u_{a})=\xi(a)$ for all unitary generators $u_{a}$ and $a \in A$. In particular we have that $H_{\xi}$ is a state on $C^{*}(A)$. The condition that $\ker(\Phi) \cap \mathcal{G}_{x}^{x} = \{x\}$ implies that $\Phi: \mathcal{G}_{x}^{x} \to A$ is an injective group homomorphism for each $x \in \mathcal{G}^{(0)}$, which gives us an injective unital $*$-homomorphism $\iota_{x}: C^{*}(\mathcal{G}_{x}^{x}) \to C^{*}(A)$ satisfying $\iota_{x}(u_{g})=u_{\Phi(g)}$ for all $g \in \mathcal{G}_{x}^{x}$. For each $x \in \mathcal{G}^{(0)}$ we define a state $\psi_{x} : = H_{\xi} \circ \iota_{x}$ on $C^{*}(\mathcal{G}_{x}^{x})$ and claim that $\{\psi_{x}\}_{x \in \mathcal{G}^{(0)}}$ is a $\mu$-measurable field of states. It suffices to prove that
$$
\mathcal{G}^{(0)} \ni x \to \sum_{g \in \mathcal{G}_{x}^{x}} f(g) \psi_{x}(u_{g})
$$
is $\mu$-measurable for $f \in C_{c}(\mathcal{G})$ with $\text{supp}(f) \subseteq  W \subseteq \overline{W} \subseteq U \subseteq \Phi^{-1}(\{a\})$ where $W$ is open, $\overline{W}$ is compact, $U$ is an open bisection and $a \in A$. The set $N:=\{g\in \overline{W} \ | \ s(g)=r(g)\}$ is compact in $\mathcal{G}$, so $r(N)$ is closed in $\mathcal{G}^{(0)}$. However
$$
\mathcal{G}^{(0)} \setminus r(N) \ni x \to  \sum_{g \in \mathcal{G}_{x}^{x}} f(g) \psi_{x}(u_{g})=0
$$
while for $x \in r(N)$ we have:
$$
\sum_{g \in \mathcal{G}_{x}^{x}} f(g) \psi_{x}(u_{g}) = \sum_{g \in \mathcal{G}_{x}^{x}} f(g) \xi(\Phi(g))=  f(r_{\overline{W}}^{-1}(x)) \xi(a) .
$$
So since $r(N) \ni x \to f(r_{\overline{W}}^{-1}(x)) \xi(a)$ is continuous $\{\psi_{x}\}_{x \in \mathcal{G}^{(0)}}$ is a $\mu$-measurable field. For any $x \in \mathcal{G}^{(0)}$ and all $g \in \mathcal{G}_{x}^{x}$ and $h\in\mathcal{G}_{x}$ we have $\psi_{r(h)}(u_{hgh^{-1}})=\xi(\Phi(hgh^{-1})) = \psi_{x}(u_{g})$, so $\{\psi_{x}\}_{x \in \mathcal{G}^{(0)}}$ satisfies \eqref{htn}. To prove $\omega_{1}$ is extremal assume that $\omega_{1}=\lambda \varphi'+(1-\lambda)\tilde{\varphi}$ with $\varphi', \tilde{\varphi} \in \Delta_{\mu}$. Now letting $\{\psi'_{x}\}_{x \in \mathcal{G}^{(0)}}$, $\{\tilde{\psi}_{x}\}_{x \in \mathcal{G}^{(0)}}$ be $\mu$-measurable fields corresponding to respectively $\varphi'$ and $\tilde{\varphi}$, then:
$$
\omega_{1}(f) = \lambda \varphi'(f)+(1-\lambda)\tilde{\varphi}(f)= \int_{\mathcal{G}^{(0)}} \sum_{g \in \mathcal{G}_{x}^{x}} f(g) (\lambda \psi'_{x}+(1-\lambda)\tilde{\psi}_{x})(u_{g}) \ d\mu(x) .
$$
Using the uniqueness result of Neshveyev we get that $\lambda \psi'_{x}+(1-\lambda)\tilde{\psi}_{x}=\psi_{x}$ for $\mu$-almost all $x$ in $\mathcal{G}^{(0)}$. However $\psi_{x}= H_{1} \circ \iota_{x}$ is multiplicative on an abelian $C^{*}$-algebra, giving that it is a pure-state by Corollary 2.3.21 in \cite{BR}. So $\psi_{x}'$ and $\tilde{\psi_{x}}$ has to be equal to $\psi_{x}$ for $\mu$-almost all $x$, and hence $\tilde{\varphi} = \omega_{1}=\varphi'$ which proves the proposition.
\end{proof}

\section{Symmetries of the KMS simplex}
We now combine the results from the last two sections to obtain a description of the extremal points of the simplex of $\beta$-KMS states for $\beta \neq 0$. Throughout this section we again consider a groupoid $\mathcal{G}$ with an abelian valued homomorphism $\Phi:\mathcal{G} \to A$ as in Definition \ref{d2}.

\begin{thm}\label{t34}
Let $c: \mathcal{G} \to \mathbb{R}$ be a continuous groupoid homomorphism and $\beta \in \mathbb{R} \setminus \{0\}$. Then any extremal $\beta$-KMS state $\omega$ for $\alpha^{c}$ is on the form:
\begin{equation}\label{kms}
\omega(f) = \int_{\mathcal{G}^{(0)}} \sum_{g \in \mathcal{G}_{x}^{x}} f(g) \xi(\Phi(g)) \ d\mu(x) \qquad \forall f \in C_{c}(\mathcal{G})
\end{equation}
where $\mu\in \partial \tilde{\Delta}$ and $\xi \in \hat{A}$. Conversely any state on this form is extremal.
\end{thm}
\begin{proof}
Let $\omega$ be given by the pair $(\mu , \{\psi_{x}\}_{x \in \mathcal{G}^{(0)}})$ as in Theorem 1.3 in \cite{N}, then $\mu \in \partial \tilde{\Delta}$ by Lemma \ref{l24}. Constructing $\omega_{1}$ using $\mu$ as in Proposition \ref{t22} then $\omega_{1}$ is extremal in $\Delta_{\mu}$, and Theorem \ref{opt} then implies that $\omega=\omega_{1} \circ \Psi_{\xi}$ for some $\xi$. Since $\omega_{1} \circ \Psi_{\xi}$ is equal to $\omega_{\xi}$ from Proposition \ref{t22} this proves the formula. Conversely the state in \eqref{kms} equals $\omega_{1} \circ \Psi_{\xi}$ and hence it is extremal by Proposition \ref{t22}.
\end{proof}

We will say that an extremal KMS state $\omega$ is given by a pair $(\mu, \xi)\in \partial\tilde{\Delta} \times \hat{A}$, when $\omega$ can be written as in \eqref{kms}. The representation of the extremal KMS state is not necessarily unique: If a state is given by a pair $(\mu, \xi)$ and a pair $(\mu', \xi')$ then clearly $\mu=\mu'$, but we might not have $\xi=\xi'$. In the following Theorem we will address this issue.

\begin{thm}\label{tmain}
Let $c: \mathcal{G} \to \mathbb{R}$ be a continuous groupoid homomorphism and $\beta \in \mathbb{R} \setminus \{0\}$. Let $\mu \in \partial \tilde{\Delta}$ and let $\omega$ be the extremal $\beta$-KMS state for $\alpha^{c}$ given by the pair $(\mu, 1)$. Then:
$$
N:=\{ \xi \in \hat{A} \ | \ \omega \circ \Psi_{\xi} = \omega \}
$$
is a closed subgroup in $\hat{A}$. Consider the subgroup:
$$
B:=N^{\perp}=\{a \in A \ | \  \xi(a)=1 \text{ for all } \xi \in N\} \subseteq A .
$$
Then the following is true:
\begin{enumerate}
\item For any subgroup $C \subseteq A$ the set 
$$
X(C):=\left\{ x \in \mathcal{G}^{(0)} \ | \ \Phi(\mathcal{G}_{x}^{x})=C \right\}
$$
is a Borel set in $\mathcal{G}^{(0)}$, and:
$$
\mu\left( X(C)\right)=
\begin{cases}
1& \text{ if } C=B \\
0& \text{ else.}
\end{cases}
$$

\item $\Delta_{\mu}$ is a simplex and $\hat{B} \simeq \hat{A} /N$ acts transitively and freely on $\partial \Delta_{\mu}$. This gives rise to a homeomorphism:
\begin{equation}\label{homeo}
\hat{B} \ni \xi \to \left[ f \to \int_{X(B)} \sum_{g \in \mathcal{G}_{x}^{x}} f(g) \xi(\Phi(g)) \ d\mu(x) \right] \in\partial \Delta_{\mu} .
\end{equation}
\end{enumerate}
\end{thm}

\begin{proof}
Checking that $N$ is a closed subgroup is straightforward.
To prove $(1)$, we first claim that $X(a):=\left\{x\in \mathcal{G}^{(0)} \ | \ a \in \Phi(\mathcal{G}_{x}^{x})  \right \}$ is a Borel set in $\mathcal{G}^{(0)}$ for all $a \in A$. Since $A$ is countable $\{g \in \Phi^{-1}(\{a\}) \ | \ r(g) =s(g) \}$ is closed in $\mathcal{G}$, so since $\mathcal{G}$ is second countable and \'etale this implies that $r(\{g \in \Phi^{-1}(\{a\}) \ | \ r(g) =s(g) \})=X(a)$ is Borel. Clearly $X(a)$ is an invariant set, so if $\mu(X(a)) \in ]0,1[$ then $X(a)^{C}$ would be an invariant Borel set with $\mu(X(a)^{C})\in]0,1[$, which would imply that $\mu$ could be written as a convex combination of two elements in $\tilde{\Delta}$. However $\mu \in \partial \tilde{\Delta}$, so $\mu(X(a))=0$ or $\mu(X(a))=1$. If $\mu(X(a))=1$ and $\mu(X(b))=1$ then $\mu(X(a)\cap X(b))=1$, so since $X(a) \cap X(b) \subseteq X(ab)$ and $X(a)=X(a^{-1})$ we have that:
$$
D:=\{a \in A \ | \ \mu(X(a))=1\}
$$
is a subgroup of $A$. For any subgroup $C$ of $A$ we can write:
$$
X(C)= \left(\bigcap_{c \in C} X(c) \right) \setminus \left(  \bigcup_{a \in A \setminus C} X(a) \right)
$$
and hence $X(C)$ is Borel. From this equality it also follows that $\mu(X(D))=1$. Since $X(D)\cap X(C)= \emptyset$ for subgroups $C \neq D$, this implies $\mu(X(C))=0$ when $C \neq D$. By definition of $N$ we have, using the notation of the proof of Proposition \ref{t22}, that $\xi \in N$ if and only if $H_{1} \circ \iota_{x}=H_{\xi}\circ \iota_{x}$ for $\mu$ almost all $x \in \mathcal{G}^{(0)}$, so if and only if $D \subseteq \text{Ker}(\xi)$. However $D \subseteq \text{Ker}(\xi)$ if and only if $\xi \in D^{\perp}$, so combined we get that $D^{\perp}=N$, and hence $D=(D^{\perp})^{\perp}=N^{\perp}=B$.

\smallskip

To prove $(2)$ notice first that the map that sends $\phi B^{\perp} \in \hat{A}/B^{\perp}$ to $\phi|_{B} \in \widehat{B}$ is an isomorphism by Theorem 2.1.2 in \cite{Ru}, so since $N=B^{\perp}$ this proves $\hat{A}/N \simeq \hat{B}$. Since $\mu$ is extremal in $\tilde{\Delta}$ it follows by Theorem \ref{t34} that every $\psi \in \partial \Delta_{\mu}$ is on the form $\omega \circ \Psi_{\xi}$ for some $\xi \in \hat{A}$, so by definition of $N$ we can define a transitive and free action of $\hat{A}/N$ on $\partial \Delta_{\mu}$ by:
$$
\hat{A}/N \times \partial \Delta_{\mu} \ni (\xi N, \psi) \to \psi \circ \Psi_{\xi} \in \partial \Delta_{\mu} .
$$
So the map in \eqref{homeo} is a bijection, and since functions $f \in C_{c}(\mathcal{G})$ supported in some set $\Phi^{-1}(\{a\})$, $a \in A$, spans $C_{c}(\mathcal{G})$ it follows that the map is continuous, and hence a homeomorphism since $\hat{B}$ is compact. To see that $\Delta_{\mu}$ is a simplex, let $\mu_{1}$ and $\mu_{2}$ be two different maximal regular Borel probability measures on $\Delta_{\mu}$, and assume for contradiction that they have the same barycenter. Then $\int_{\Delta_{\mu}} \gamma(x)d\mu_{1}(\gamma)=\int_{\Delta_{\mu}} \gamma(x)d\mu_{2}(\gamma)$ for all $x \in C^{*}(\mathcal{G})$. Since $\Delta$ is metrizable $\mu_{1}$ and $\mu_{2}$ are supported on $\partial \Delta_{\mu}$, so we consider them as measures on $\hat{B}$. It follows from Stone-Weierstrass that the span of $\{\text{ev}_{b} \ | \ b \in B\}$ is a dense subalgebra of $C(\hat{B})$, so there exist a $b \in B$ with $\int_{\hat{B}} \xi(b) d \mu_{1}(\xi) \neq \int_{\hat{B}} \xi(b) d \mu_{2}(\xi)$. Since $\mathcal{G}$ is sigma-compact and $\Phi^{-1}(\{b\})$ is clopen, there is an increasing sequence of positive functions $f_{n}\in C_{c}(\mathcal{G})$ that converges pointwise to $1_{\Phi^{-1}(\{b\})}$, and hence the functions $x \to \sum_{g \in \mathcal{G}_{x}^{x}} f_{n}(g)$ increases pointwise to a function $f'$ with $f'=1$ on $X(B)$. Using monotone convergence there is a $f \in C_{c}(\mathcal{G})$ with $\text{supp}(f) \subseteq \Phi^{-1}(\{b\})$ and $\omega(f)=\int_{X(B)}\sum_{g \in \mathcal{G}_{x}^{x}} f(g) \ d \mu (x) \neq 0$, and hence for $i=1,2$ we have:
\begin{align*}
\int_{ \Delta_{\mu}} \gamma(f)d\mu_{i}(\gamma) = \int_{\hat{B}} \omega(\Psi_{\xi}(f)) d\mu_{i} (\xi) = \omega(f) \int_{\hat{B}} \xi(b) d\mu_{i} (\xi)
\end{align*}
a contradiction. Hence $\Delta_{\mu}$ is a simplex.
\end{proof}

\begin{obs}
This Lemma should be compared with Proposition 11.5 in \cite{aHLRS}. In \cite{aHLRS} the authors analyse the KMS states on the Cuntz-Krieger algebras of finite strongly connected higher-rank graphs, which are $C^{*}$-algebras of groupoids satisfying Definition \ref{d2}, see section \ref{s63} below or section 12 in \cite{aHLRS}. Letting $c$ be the continuous groupoid homomorphism giving rise to what the authors call the preferred dynamics, Lemma 12.1 in \cite{aHLRS} implies that there is exactly one $e^{-c\cdot 1}$-quasi-invariant measure and that the subgroup $B$ described in Theorem \ref{tmain} is the subgroup $\text{Per}(\Lambda)$, see Proposition 5.2 in \cite{aHLRS} for the definition of $\text{Per}(\Lambda)$. Then $(2)$ in our Theorem \ref{tmain} becomes Proposition 11.5 in \cite{aHLRS}.
\end{obs}

Theorem 1.3 in \cite{N} is very useful for giving a concrete description of the KMS states when either the groupoids involved only have countably many points in the unit space with non-trivial isotropy, or when it is possible to prove that all KMS states factors through the conditional expectation $P$. To illustrate how Theorem \ref{tmain} can be used in more complex cases, we will now use it to analyse the KMS states for Cuntz-Krieger $C^{*}$-algebras of finite higher-rank graphs without sources, where neither of the two classical approaches suffices.

\section{Background on higher-rank graphs}

\subsection{The Cuntz-Krieger $C^{*}$-algebras of higher-rank graphs} \label{s61}
For $k \in \mathbb{N}$ we always denote the standard basis for $\mathbb{N}^{k}$ by $\{e_{1}, e_{2}, \dots, e_{k}\}$, and for $n,m \in \mathbb{N}^{k}$ we write $n \leq m$ if $n_{i} \leq m_{i}$ for all $i=1,2, \dots , k$, and $n \vee m $ for the vector in $\mathbb{N}^{k}$ with $(n\vee m)_{i}=\max\{n_{i}, m_{i}\}$ for all $i$. A higher-rank graph of rank $k \in \mathbb{N}$ is a pair $(\Lambda , d)$ consisting of a countable small category $\Lambda$ and a functor $d: \Lambda \to \mathbb{N}^{k}$ which satisfies the factorisation property: for every $\lambda \in \Lambda$ and every decomposition $d(\lambda)=n+m$, $n,m \in \mathbb{N}^{k}$, there exists unique $\mu, \nu \in \Lambda$ with $d(\mu)=n$, $d(\nu)=m$ and $\lambda = \mu \nu$. For all $n \in \mathbb{N}^{k}$ we write $\Lambda^{n}:=d^{-1}(\{n\})$ and we identify the objects of $\Lambda$ with $\Lambda^{0} \subseteq \Lambda$ and call these vertices. Elements of $\Lambda$ are referred to as paths, and we use the range and source maps $r,s: \Lambda \to \Lambda^{0}$ to make sense of the start $s(\lambda)$ and the end $r(\lambda)$ of paths $\lambda$ in  $\Lambda$. For $0\leq l \leq n \leq m$ and $\lambda \in \Lambda^{m}$ we denote by $\lambda(0,l)\in \Lambda^{l}$, $\lambda(l,n) \in \Lambda^{n-l}$ and $\lambda(n,m)\in \Lambda^{m-n}$ the unique paths with $\lambda=\lambda(0,l)\lambda(l,n)\lambda(n,m)$. We often abbreviate and write $\Lambda$ for a higher-rank graph of rank $k$ and simply call it a $k$-graph. For any $X,Y \subseteq \Lambda$ we write $XY$ for the set:
$$
XY:=\{ \mu\lambda \ | \ \mu \in X, \lambda \in Y \text{ and } s(\mu)=r(\lambda) \}
$$
and we use variations on this theme to define sets throughout the next sections. We say that a $k$-graph $\Lambda$ is \emph{finite} if $\Lambda^{n}$ is finite for all $n \in \mathbb{N}^{k}$ and \emph{without sources} if $v\Lambda^{n} \neq \emptyset$ for all $n \in \mathbb{N}^{k}$ and $v \in \Lambda^{0}$. We can define a relation on $\Lambda^{0}$ by defining $v \leq w$ if $v \Lambda w \neq \emptyset$, i.e. if there is a path starting in $w$ and ending in $v$. This gives an equivalence relation $\sim$ on $\Lambda^{0}$ by defining $v \sim w$ if $v \leq w$ and $w \leq v$. We call these equivalence classes \emph{components}, and more specifically we call a component $C$ \emph{trivial} if $C\Lambda C=\{v\}$ for some $v \in \Lambda^{0}$ and \emph{non-trivial} if this is not the case. The relation $\leq$ descends to a partial order on the set of components, i.e. $C \leq D$ if $C\Lambda D \neq \emptyset$. For sets $V \subseteq \Lambda^{0}$ we define the \emph{closure} of $V$ to be $\overline{V}=\{ w \in \Lambda^{0}\ | \ w \Lambda V \neq \emptyset \}$ and the \emph{hereditary closure} to be $\widehat{V}=\{ w \in \Lambda^{0}\ | \ V \Lambda w \neq \emptyset \}$. For any set $S$ that is closed, hereditary closed or a component we can define a new higher-rank graph $(\Lambda_{S}, d)$ where $\Lambda_{S}=S \Lambda S$. A graph is called \emph{strongly connected} if $v\Lambda w \neq \emptyset$ for all $v,w \in \Lambda^{0}$, and we notice that $(\Lambda_{C},d)$ is a strongly connected graph for all components $C$ of $\Lambda^{0}$.

For a finite $k$-graph $\Lambda$ we can define the $\Lambda^{0}\times \Lambda^{0}$ vertex matrices $A_{1}, \dots , A_{k}$ with entries $A_{i}(v,w)=\lvert v\Lambda^{e_{i}} w\rvert$. The factorisation property implies that these commute, and defining $A^{n}=\prod_{i=1}^{k} A_{i}^{n_{i}}$ for each $n \in \mathbb{N}^{k}$ one can prove that $A^{n}(v,w)=\lvert v \Lambda^{n} w\rvert$.

\begin{defn}
Let $\Lambda$ be a finite $k$-graph without sources. A Cuntz-Krieger $\Lambda$-family is a set of partial isometries $\{t_{\lambda} \ | \ \lambda \in \Lambda \}$ in a $C^{*}$-algebra satisfying:
\begin{itemize}
\item[(CK1)] $\{t_{v} \ | \ v \in \Lambda^{0} \}$ is a set of mutually orthogonal projections.
\item[(CK2)] $t_{\lambda} t_{\gamma}=t_{\lambda\gamma}$ for all $\lambda, \gamma \in \Lambda$ with $r(\gamma)=s(\lambda)$.
\item[(CK3)] $t_{\lambda}^{*} t_{\lambda}=t_{s(\lambda)}$ for all $\lambda \in \Lambda$.
\item[(CK4)] $t_{v}=\sum_{\lambda \in v\Lambda^{n}} t_{\lambda}t_{\lambda}^{*}$ for all $v \in \Lambda^{0}$ and $n \in \mathbb{N}^{k}$.
\end{itemize}
We let $C^{*}(\Lambda)$ denote the $C^{*}$-algebra generated by a universal Cuntz-Krieger $\Lambda$-family.
\end{defn}
To ease notation we define the projection $p_{v}:=t_{v}$ for all $v \in \Lambda^{0}$, and we remind the reader that $C^{*}(\Lambda)=\overline{\text{span}} \{t_{\lambda} t_{\gamma}^{*} \ | \ \lambda ,  \gamma \in \Lambda, \ s(\lambda) = s(\gamma)\}$ and that the universal property of $C^{*}(\Lambda)$ guarantees a strongly continuous action $\gamma: \mathbb{T}^{k} \to \text{Aut}(C^{*}(\Lambda))$ by specifying that:
\begin{equation} \label{egauge}
\gamma_{z}(t_{\lambda})=z^{d(\lambda)} t_{\lambda} =\prod_{i=1}^{k} z_{i}^{d(\lambda)_{i}} t_{\lambda} \qquad \forall z \in  \mathbb{T}^{k} , \quad \forall \lambda \in \Lambda .
\end{equation}
By setting:
$$
\Lambda^{\text{min}}(\lambda, \gamma):= \{ (\delta, \nu) \in \Lambda \times \Lambda \ | \ \lambda\delta = \gamma\nu \text{ and } d(\lambda\delta) = d(\lambda) \vee d(\gamma)  \}
$$
for any $\lambda, \gamma \in \Lambda$, we furthermore have the equality:
\begin{equation}\label{e62}
t_{\lambda}^{*} t_{\gamma} = \sum_{(\delta, \nu) \in \Lambda^{\text{min}}(\lambda, \gamma)} t_{\delta} t_{\nu}^{*} .
\end{equation}

\subsection{KMS states on Cuntz-Krieger algebras of higher-rank graphs}
Let $\Lambda$ be a finite $k$-graph without sources. For any $r \in \mathbb{R}^{k}$ we can define a map $\mathbb{R} \ni t \to (e^{itr_{1}}, \dots , e^{itr_{k}}) \in \mathbb{T}^{k}$. Composing this map with the action $\gamma$ from \eqref{egauge} yields a continuous one-parameter group $\{\alpha^{r}_{t}\}_{t \in \mathbb{R}}$ satisfying:
$$
\alpha_{t}^{r}(t_{\lambda}t_{\gamma}^{*}) = \prod_{l=1}^{k} (e^{itr_{l}})^{d(\lambda)_{l}}  \prod_{l=1}^{k} (e^{-itr_{l}})^{d(\gamma)_{l}}    t_{\lambda}t_{\gamma}^{*} =  e^{it r \cdot ( d(\lambda)-d(\gamma))} t_{\lambda}t_{\gamma}^{*}
$$
for all $\lambda, \gamma \in \Lambda$. We are interested in determining the $\beta$-KMS states for all $\beta \in \mathbb{R}$ and all $C^{*}$-dynamical systems $(C^{*}(\Lambda), \alpha^{r})$ where $r \in \mathbb{R}^{k}$, and for this it suffices to check the KMS condition on pairs of elements on the form $t_{\lambda}t_{\gamma}^{*}$ with $\lambda, \gamma \in \Lambda$.

\subsection{The path groupoid for a finite higher-rank graph without sources} \label{s63} For a finite $k$-graph $\Lambda$ without sources we can realise $C^{*}(\Lambda)$ as a groupoid $C^{*}$-algebra. To do this, we first need to introduce the infinite path space $\Lambda^{\infty}$ of $\Lambda$. The standard example of a $k$-graph $\Omega_{k}$ is constructed by considering morphisms:
$$
\Omega_{k}: = \{(n,m)\in \mathbb{N}^{k} \times \mathbb{N}^{k} \ | \ n \leq m\}
$$
and objects $\Omega_{k}^{0}:=\mathbb{N}^{k}$ and then defining $s(n,m)=m$, $r(n,m)=n$, $d(n,m)=m-n$ and $(n,m) (m,q)=(n,q)$. An infinite path in the $k$-graph $\Lambda$ is then a functor $x: \Omega_{k} \to \Lambda$ that intertwines the degree maps, and we denote the set of infinite paths in $\Lambda$ by $\Lambda^{\infty}$. Defining for each $\lambda \in \Lambda$ a set $Z(\lambda)=\{x \in \Lambda^{\infty} \ | \ x(0,d(\lambda))=\lambda \}$ we get a basis $\{Z(\lambda)\}_{\lambda \in \Lambda}$ of compact and open sets, making $\Lambda^{\infty}$ a second countable compact Hausdorff space. For each $p \in \mathbb{N}^{k}$ we can define a continuous map $\sigma^{p} : \Lambda^{\infty} \to \Lambda^{\infty}$ by setting $\sigma^{p}(x)$ to be the infinite path $\sigma^{p}(x)(n,m)=x(n+p, m+p)$ for all $(n,m)\in \Omega_{k}$, and for any $p,q \in \mathbb{N}^{k}$ and $x \in \Lambda^{\infty}$ we then have that $\sigma^{p}(\sigma^{q}(x))=\sigma^{p+q}(x)=\sigma^{q}(\sigma^{p}(x))$. Setting $r(x)=x((0,0))$ for $x \in \Lambda^{\infty}$ we can compose  $\lambda \in \Lambda$ and $x \in \Lambda^{\infty}$ when $r(x)=s(\lambda)$ to get a new infinite path $\lambda x \in \Lambda^{\infty}$. Using $\Lambda^{\infty}$ we can now obtain the path groupoid by defining:
$$
\mathcal{G} = \left\{ (x,m-n,y)\in \Lambda^{\infty}\times \mathbb{Z}^{k} \times \Lambda^{\infty} \ | \ m,n \in \mathbb{N}^{k} \text{ and } \sigma^{m}(x)=\sigma^{n}(y) \right \}
$$
one can check that this is in fact a groupoid when defining composition as:
$$
(x,a,y)(y,b,z)=(x,a+b,z)
$$ 
and inversion by $(x,a,y)^{-1}=(y, -a , x)$ and we then obtain range and source maps satisfying $r(x,a,y)=(x,0,x)$ and $s(x,a,y)=(y,0,y)$. 
The groupoid $\mathcal{G}$ becomes a locally compact second countable Hausdorff \'etale groupoid when we consider a basis $\{Z(\lambda, \gamma) \ | \  \lambda, \gamma \in \Lambda , \ s(\lambda)=s(\gamma)  \}$ where:
$$
Z(\lambda, \gamma):= \{ (x,d(\lambda)-d(\gamma),z) \in \mathcal{G} \ | \ x \in Z(\lambda), \ z \in Z(\gamma), \ \sigma^{d(\lambda)}(x)=\sigma^{d(\gamma)}(z)    \} .
$$
We can therefore consider the groupoid $C^{*}$-algebra $C^{*}(\mathcal{G})$, and it follows from Corollary 3.5 in \cite{KP} that $ C^{*}(\Lambda) \simeq C^{*}(\mathcal{G}) $ under an isomorphism that maps $t_{\lambda}t_{\gamma}^{*}$ to $1_{Z(\lambda, \gamma)}$. Since $\mathcal{G}^{(0)} \simeq \Lambda^{\infty}$ we will identify the two spaces $C(\Lambda^{\infty})$ and $C(\mathcal{G}^{(0)})$. The action introduced in \eqref{egauge} is then the same as the gauge-action introduced in equation \eqref{e2}, and the continuous one-parameter group $\{\alpha^{r}_{t}\}_{t \in \mathbb{R}}$ obtained using a vector $r \in \mathbb{R}^{k}$ is the same as the one obtained by considering the continuous groupoid homomorphism $c_{r}: \mathcal{G} \to \mathbb{R}$ given by $c_{r}(x,n,y):=r \cdot n$.

\section{Harmonic vectors and KMS states}
In this section we start our analysis of the KMS states by describing the gauge-invariant KMS states. To do this, we will first describe a bijective correspondence between the gauge-invariant KMS states and certain \emph{harmonic vectors} over $\Lambda^{0}$:

\begin{defn}
Let $\Lambda$ be a finite $k$-graph without sources, $\beta \in \mathbb{R}$ and $r \in \mathbb{R}^{k}$. If $\psi \in [0, \infty[^{\Lambda^{0}}$ is a vector of unit 1-norm, i.e. $\sum_{v} |\psi_{v}| =1 $, and $\psi$ satisfies that:
$$
A_{i} \psi=e^{\beta r_{i}} \psi \qquad \text{for all  } i=1,2, \dots , k
$$
we call $\psi$ a $\beta$-harmonic vector for $\alpha^{r}$.
\end{defn}

\begin{lemma}\label{l72}
Let $\Lambda$ be a finite $k$-graph without sources, $\beta \in \mathbb{R}$ and $r \in \mathbb{R}^{k}$. Let $\omega$ be a $\beta$-KMS state for $\alpha^{r}$, then the vector:
$$
\{ \omega(p_{v})\}_{v \in \Lambda^{0}}
$$
is a $\beta$-harmonic vector for $\alpha^{r}$.
\end{lemma}
\begin{proof}
Set $\psi_{w}=\omega(p_{w})$ for all $w \in \Lambda^{0}$, then clearly $\psi_{w} \in [0, \infty [^{\Lambda^{0}}$ is of unit 1-norm. Using $(CK4)$ we have for each $i \in \{1, 2, \dots, k\}$ and $v \in \Lambda^{0}$ that:
\begin{align*}
\psi_{v}&=\omega(p_{v})=\sum_{\lambda \in v\Lambda^{e_{i}}}\omega( t_{\lambda}t_{\lambda}^{*}) =\sum_{\lambda \in v\Lambda^{e_{i}}}\omega( t_{\lambda}^{*}\alpha^{r}_{i\beta}(t_{\lambda})) =\sum_{\lambda \in v\Lambda^{e_{i}}} e^{-\beta r\cdot d(\lambda)}\omega( t_{\lambda}^{*}t_{\lambda})\\
&=e^{-\beta r_{i}} \sum_{w \in \Lambda^{0}}\sum_{\lambda \in v\Lambda^{e_{i}}w} \omega( p_{w}) 
=e^{-\beta r_{i}} \sum_{w \in \Lambda^{0}} A_{i}(v,w)\psi_{w} =e^{-\beta r_{i}} (A_{i}\psi)_{v} 
\end{align*}
proving the Lemma.
\end{proof}
Inspired by Proposition 8.1 in \cite{aHLRS} we can now associate a measure to a $\beta$-harmonic vector.
\begin{prop}
Let $\Lambda$ be a finite $k$-graph without sources, $\beta \in \mathbb{R}$ and $r \in \mathbb{R}^{k}$. Let $\psi$ be a $\beta$-harmonic vector for $\alpha^{r}$, then there exists a unique Borel probability measure $M_{\psi}$ on $\Lambda^{\infty}$ satisfying:
$$
M_{\psi}(Z(\lambda))=e^{-\beta r\cdot d(\lambda)} \psi_{s(\lambda)} \qquad \forall \lambda \in \Lambda .
$$

\end{prop}

\begin{proof}
For all $m,n \in \mathbb{N}^{k}$ with $m \leq n$ we define $\pi_{m,n}: \Lambda^{n} \to \Lambda^{m}$ by $\pi_{m,n}(\lambda)=\lambda(0,m)$. Since $\Lambda$ is without sources the maps $\pi_{m,n}$ are surjective. Giving $\Lambda^{n}$ the discrete topology for each $n \in \mathbb{N}^{k}$, it follows that $(\Lambda^{m}, \pi_{m,n})$ is an inverse system of compact topological spaces and continuous surjective maps, and hence we get a topological space $\varprojlim(\Lambda^{m}, \pi_{m,n})$. It is straightforward to check that 
$$
\Lambda^{\infty}\ni x \to \{x(0,m)\}_{m \in \mathbb{N}^{k}} \in \varprojlim(\Lambda^{m}, \pi_{m,n})
$$
is a homeomorphism. For each $m\in \mathbb{N}^{k}$ we now define a measure $M_{m}$ on $\Lambda^{m}$ by:
$$
M_{m}(S)=e^{-\beta r \cdot m} \sum_{\lambda \in S} \psi_{s(\lambda)} \qquad \text{for } S \subseteq \Lambda^{m} .
$$
For $m \leq n$ and $\lambda \in \Lambda^{m}$ we have:
\begin{align*}
&M_{n} (\pi^{-1}_{m,n}(\{\lambda\})) =e^{-\beta r \cdot n} \sum_{\nu \in \pi^{-1}_{m,n}(\{\lambda\})} \psi_{s(\nu)}
=e^{-\beta r \cdot n} \sum_{\alpha \in s(\lambda)\Lambda^{n-m}} \psi_{s(\alpha)} \\
&=e^{-\beta r \cdot n} \sum_{w \in \Lambda^{0}}A^{n-m}(s(\lambda),w)\psi_{w}
=e^{-\beta r \cdot n} (A^{n-m}\psi)_{s(\lambda)}=e^{-\beta r \cdot m}\psi_{s(\lambda)}\\
&= M_{m}(\{\lambda\}) .
\end{align*}
Combining this calculation with Lemma 5.2 in \cite{aHKR2} gives us a regular Borel measure $M_{\psi}$ on $\Lambda^{\infty}$ such that:
$$
M_{\psi}(Z(\lambda))=M_{m}(\{\lambda\}) = e^{-\beta r \cdot m} \psi_{s(\lambda)}=e^{-\beta r \cdot d(\lambda)} \psi_{s(\lambda)}
$$
for $\lambda \in \Lambda^{m}$. Since $\psi$ is of unit $1$-norm $M_{\psi}$ is a probability measure, and $M_{\psi}$ is clearly unique. 
\end{proof}

For each $M_{\psi}$ we define a state $\omega_{\psi}$ on $C^{*}(\Lambda)$ by:
$$
\omega_{\psi}(a) = \int_{\Lambda^{\infty}} P(a) \ d M_{\psi} .
$$

\begin{thm}\label{tvector}
Assume $\Lambda$ is a finite $k$-graph without sources, $\beta \in \mathbb{R}$ and $r \in \mathbb{R}^{k}$. The map $ \psi \to \omega_{\psi}$ is an affine bijection from the $\beta$-harmonic vectors for $\alpha^{r}$ to the gauge-invariant $\beta$-KMS states for $\alpha^{r}$.
\end{thm}

\begin{proof}
Let $\psi$ be a $\beta$-harmonic vector for $\alpha^{r}$ and let $M_{\psi}$ be the corresponding Borel probability measure on $\Lambda^{\infty}$. Since the gauge-action fixes $C(\Lambda^{\infty}) $ it follows that $\omega_{\psi}$ is a gauge-invariant state. We will now argue that it is a $\beta$-KMS state for $\alpha^{r}$, so let $\lambda, \gamma, \delta, \epsilon \in \Lambda$ with $s(\lambda)=s(\gamma)$ and $s(\delta)=s(\epsilon)$ with $d(\delta) \geq d(\gamma)$ and $d(\epsilon) \geq d(\lambda)$. Using equation \eqref{e62} we have that:
$$
\omega_{\psi}(t_{\lambda}t_{\gamma}^{*}t_{\delta}t_{\epsilon}^{*}) 
= \sum_{(\eta, \nu)\in \Lambda^{\text{min}}(\gamma ,\delta)} \omega_{\psi}(t_{\lambda\eta} t_{\epsilon \nu}^{*})
= \sum_{(\eta, \nu)\in \Lambda^{\text{min}}(\gamma ,\delta) , \ \lambda\eta= \epsilon \nu} e^{-\beta r \cdot d(\lambda \eta)} \psi_{s(\eta)} .
$$
On the other hand:
\begin{align*}
\omega_{\psi}(t_{\delta}t_{\epsilon}^{*}\alpha_{i \beta}^{r}(t_{\lambda}t_{\gamma}^{*})) 
&= e^{-\beta r \cdot(d(\lambda)-d(\gamma))} \sum_{(\kappa, \tau)\in \Lambda^{\text{min}}(\epsilon ,\lambda)} \omega_{\psi}(t_{\delta \kappa} t_{\gamma \tau}^{*}) \\
&= e^{-\beta r \cdot(d(\lambda)-d(\gamma))} \sum_{(\kappa, \tau)\in \Lambda^{\text{min}}(\epsilon ,\lambda) , \  \delta \kappa = \gamma \tau} e^{-\beta r \cdot d(\gamma \tau)} \psi_{s(\tau)} \\
&=  \sum_{(\kappa, \tau)\in \Lambda^{\text{min}}(\epsilon ,\lambda) , \  \delta \kappa = \gamma \tau} e^{-\beta r \cdot d(\lambda \tau)} \psi_{s(\tau)} 
\end{align*}
Now we claim that $(\eta, \nu) \to ( \nu, \eta)$ is a bijection from $\{(\eta, \nu)\in \Lambda^{\text{min}}(\gamma ,\delta) \ | \ \lambda\eta= \epsilon \nu \}$ to $\{ (\kappa, \tau)\in \Lambda^{\text{min}}(\epsilon ,\lambda) \ | \  \delta \kappa = \gamma \tau \}$. To see this, notice that by assumption $d(\gamma) \vee d(\delta)=d(\delta)$ and $d(\epsilon) \vee d(\lambda) = d(\epsilon)$, so $d(\nu)=0$ and:
$$
 d(\lambda \eta) =d(\epsilon \nu)= d(\epsilon)+d(\nu)=d(\epsilon)=d(\epsilon) \vee d(\lambda).
$$
So $(\nu , \eta) \in \Lambda^{\text{min}}(\epsilon , \lambda)$ and by choice $\delta \nu = \gamma \eta$, proving that the map is well defined. It is straightforward to check that it has an inverse given by $(\kappa, \tau) \to (\tau , \kappa)$ and hence it is a bijection. This implies that:
$$
 \omega_{\psi}(t_{\lambda}t_{\gamma}^{*}t_{\delta}t_{\epsilon}^{*}) 
= \sum_{(\eta, \nu)\in \Lambda^{\text{min}}(\gamma ,\delta) , \ \lambda\eta= \epsilon \nu} e^{-\beta r \cdot d(\lambda \eta)} \psi_{s(\eta)}
= \sum_{ (\kappa, \tau)\in \Lambda^{\text{min}}(\epsilon ,\lambda) , \  \delta \kappa = \gamma \tau} e^{-\beta r \cdot d(\lambda \tau)} \psi_{s(\tau)} .
$$
This proves that $\omega_{\psi}$ satisfies the $\beta$-KMS condition for such pairs $t_{\lambda}t_{\gamma}^{*}$, $t_{\delta}t_{\epsilon}^{*}$. For such a pair not necessarily satisfying $d(\delta) \geq d(\gamma)$ and $d(\epsilon) \geq d(\lambda)$, taking a large $n$ and using $(CK4)$ yield:
$$
\omega_{\psi}(t_{\lambda}t_{\gamma}^{*}t_{\delta}t_{\epsilon}^{*}) = \sum_{\upsilon \in s(\delta)\Lambda^{n}} \omega_{\psi}(t_{\lambda}t_{\gamma}^{*}t_{\delta \upsilon}t_{\epsilon \upsilon}^{*}) 
= \sum_{\upsilon \in s(\delta)\Lambda^{n}} \omega_{\psi}(t_{\delta \upsilon}t_{\epsilon\upsilon}^{*} \alpha_{i \beta}^{r}(t_{\lambda}t_{\gamma}^{*})) 
=\omega_{\psi}(t_{\delta}t_{\epsilon}^{*}\alpha_{i \beta}^{r}(t_{\lambda}t_{\gamma}^{*}))
$$
proving that $\omega_{\psi}$ is a $\beta$-KMS state for $\alpha^{r}$. 

So $\psi \to \omega_{\psi}$ is well-defined. For injectivity, notice that:
$$
\omega_{\psi} (p_{v})=M_{\psi}(Z(v))=\psi_{v}
$$
by definition of $M_{\psi}$. To prove that it is surjective, take a gauge-invariant $\beta$-KMS state for $\alpha^{r}$, say $\omega$. It follows from Lemma \ref{l72} that setting $\psi_{v}=\omega(p_{v})$ then $\psi$ is a $\beta$-harmonic vector for $\alpha^{r}$. It follows from Lemma \ref{l320} that $\omega = \omega \circ P$, so $\omega$ is given by a Borel probability measure $M$ on $\Lambda^{\infty}$. Since $\omega$ is a KMS state we have:
$$
M(Z(\lambda))=\omega(t_{\lambda}t_{\lambda}^{*}) = e^{-\beta r \cdot d(\lambda)}\omega(t_{\lambda}^{*}t_{\lambda}) = e^{-\beta r \cdot d(\lambda)}\psi_{s(\lambda)} = M_{\psi}(Z(\lambda))
$$
proving that $\omega_{\psi} = \omega$, and hence surjectivity.
\end{proof}

\subsection{Decomposition of Harmonic vectors}

To describe the gauge-invariant KMS states Theorem \ref{tvector} informs us that it is sufficient to analyse the set of harmonic vectors, which we will do in the following. It turns out, much like in the case for $1$-graphs, that the set of harmonic vectors is a finite simplex, and that the extremal points in this simplex arise from certain components in the graph, see e.g. \cite{aHLRS1} or \cite{CT} for the $1$-graph case (but be aware that \cite{CT} uses a different convention for traversing paths). The technique used in \cite{aHLRS1} required that the vertex set was ordered such that the vertex matrix was block upper diagonal, this is however difficult to do for graphs of rank $k>1$, since one has to juggle numerous vertex matrices at once. To overcome this problem we define a new matrix $A_{F}$ that incorporates all the vertex matrices as follows: Let $F$ be a finite sequence $\{a_{1}, a_{2}, \dots , a_{m}\}$ of elements in $\mathbb{N}^{k}\setminus \{0\}$ (i.e. $F \in \prod_{i=1}^{m} (\mathbb{N}^{k}\setminus \{0\})$ for some $m \in \mathbb{N}$) and set:
$$
A_{F}:=\sum_{n \in F} A^{n} = \sum_{j=1 }^{m} A^{a_{j}} .
$$
Our reason for considering $F$ as a sequence is that we allow for the same vector to occur multiple times in $F$. We call such a set $F$ \emph{well chosen} if for all $v,w \in \Lambda^{0}$, $A_{F}(v,w)>0$ if and only if $v \Lambda^{l} w \neq \emptyset$ for some $l \in \mathbb{N}^{k} \setminus \{0\}$. Since $A^{n}(v,w)=|v\Lambda^{n}w|$ it follows that there always exist a well chosen set, and if $F$ is well chosen and $S=\{b_{1}, \dots b_{q}\}$ is a finite sequence in $\mathbb{N}^{k} \setminus \{0\}$, then the concatenation 
$$
F \cup S:=\{a_{1}, \dots , a_{m}, b_{1}, \dots, b_{q}\} \in \prod_{i=1}^{m+q} (\mathbb{N}^{k}\setminus \{0\})
$$
of the two sequences is a well chosen set as well. Similar to the strongly connected graph case it turns out that there is a connection between eigenvectors for $A_{F}$ and eigenvectors for the vertex matrices $A_{i}$, $i=1, \dots ,k$, c.f. Proposition 3.1 in \cite{aHLRS}. Given a $k$-graph $\Lambda$, a $\Lambda^{0} \times \Lambda^{0}$ matrix $B$ and $S,R \subseteq \Lambda^{0}$ we will write $B^{R,S}$ for the matrix $B$ restricted to the rows $R$ and columns $S$, and when we in the following write $A_{F}^{R,S}$ we specifically mean $(A_{F})^{R,S}$. Set $B^{S}:=B^{S,S}$. For vectors $x$ we will denote the restriction to a set $S$ by $x|_{S}$.

\begin{defn}
Let $F$ be a well chosen set for a finite $k$-graph $\Lambda$ without sources. We say that a non-trivial component $C$ is $F$-harmonic if either $\overline{C} \setminus C = \emptyset$ or:
$$
\rho(A_{F}^{C}) > \rho(A_{F}^{\overline{C} \setminus C}) .
$$
We call a component $C$ positive if $\rho(A_{i}^{C})>0$ for all $i \in \{1, 2, \dots , k\}$.
\end{defn}

Notice that if $C$ is a non-trivial component and $F$ is well chosen then $A_{F}^{C}$ is a strictly positive integer matrix and hence $\rho(A_{F}^{C})>0$.

\begin{lemma}\label{lC}
Let $F$ be a well chosen set for a finite $k$-graph $\Lambda$ without sources, and let $C$ be an $F$-harmonic component. Then there exists a unique vector $x_{F}^{C} \in [0, \infty[^{\Lambda^{0}}$ of unit 1-norm that satisfies:
\begin{enumerate}
\item $A_{F}  x_{F}^{C}= \rho(A_{F}^{C}) x_{F}^{C}$
\item $(x_{F}^{C})_{v}= 0$ for $v \notin \overline{C}$.
\end{enumerate}
This vector will furthermore satisfy that $(x_{F}^{C})_{v}>0$ for all $v \in \overline{C}$ and that $x_{F}^{C}|_{C}=cx$ where $c>0$ and $x$ is the unimodular Perron-Frobenious eigenvector for $A_{F}^{C}$.
\end{lemma}

\begin{proof}
$A_{F}^{C}$ is strictly positive, so there exists a unique vector $x \in ]0,\infty[^{C}$ with unit $1$-norm such that $A_{F}^{C} x = \rho(A_{F}^{C}) x$. If $\overline{C}\setminus C \neq \emptyset$ it follows by choice of $C$ that the matrix $\rho(A_{F}^{C})^{-1} A_{F}^{\overline{C} \setminus C} $ has spectral radius strictly less than $1$. So:
$$
\left( 1^{\overline{C} \setminus C} - \rho(A_{F}^{C})^{-1}A_{F}^{\overline{C} \setminus C}  \right)^{-1} = \sum_{n=0}^{\infty}\left( \rho(A_{F}^{C})^{-1}A_{F}^{\overline{C} \setminus C}  \right)^{n} .
$$
We define a vector $x^{C} \in [0, \infty[^{\overline{C}}$ as follows; If $\overline{C}\setminus C=\emptyset$ we set $x^{C}=x$ and if $\overline{C}\setminus C \neq \emptyset$ we set:
$$
x^{C}|_{\overline{C}\setminus C} = \left( 1^{\overline{C} \setminus C} - \rho(A_{F}^{C})^{-1}A_{F}^{\overline{C} \setminus C}  \right)^{-1} (\rho(A_{F}^{C})^{-1} A_{F}^{\overline{C} \setminus C, C}) x \quad , \quad x^{C}|_{C}=x .
$$
By definition of $F$, $A_{F}^{\overline{C} \setminus C, C}$ is a matrix with strictly positive entries, so $x^{C}_{v} >0$ for all $v \in \overline{C}$. If $v \in C$ and $w \in \overline{C} \setminus C$ then $v \Lambda w =\emptyset$. So $A_{F}^{C, \overline{C} \setminus C}=0$, and this implies that $\left(A_{F}^{\overline{C}} x^{C}\right)|_{C}=A_{F}^{C}x=\rho(A_{F}^{C})x^{C}|_{C}$ and:
\begin{align*}
 &\left(A_{F}^{\overline{C}}x^{C}\right)|_{\overline{C}\setminus C}= A_{F}^{\overline{C}\setminus C} \sum_{n=0}^{\infty}\left( \rho(A_{F}^{C})^{-1} A_{F}^{\overline{C} \setminus C}  \right)^{n}(\rho(A_{F}^{C})^{-1} A_{F}^{\overline{C} \setminus C, C}) x+ A_{F}^{\overline{C}\setminus C, C}x \\
&= \rho(A_{F}^{C}) \sum_{n=0}^{\infty}\left( \rho(A_{F}^{C})^{-1}A_{F}^{\overline{C} \setminus C} \right)^{n+1}(\rho(A_{F}^{C})^{-1}A_{F}^{\overline{C} \setminus C, C}) x+ \rho(A_{F}^{C}) \left(\rho(A_{F}^{C})^{-1}A_{F}^{\overline{C}\setminus C, C} \right)x  \\
&=\rho(A_{F}^{C}) x^{C}|_{\overline{C}\setminus C}
\end{align*}
Define a vector $x_{F}^{C}\in [0, \infty[^{\Lambda^{0}}$ by setting $(x_{F}^{C})_{v}=0$ when $v \notin \overline{C}$ and $x_{F}^{C}|_{\overline{C}}$ to be the normalisation of $x^{C}$. Since $A_{F}^{\Lambda^{0} \setminus \overline{C}, \overline{C}}=0$ it then follows that:
$$
A_{F}x_{F}^{C} =\rho(A_{F}^{C}) x_{F}^{C}
$$
which proves existence.

To prove uniqueness, assume $y \in [0 , \infty[^{\Lambda^{0}}$ is of unit $1$-norm and satisfies $(1)$ and $(2)$. Then by $(2)$ and since $A_{F}^{C ,\overline{C} \setminus C}=0$ we have:
$$
\left(A_{F} y \right)|_{\overline{C}} =  A_{F}^{\overline{C}} (y|_{\overline{C}}) \qquad \left(A_{F}^{\overline{C}} (y|_{\overline{C}}) \right)|_{C} = A_{F}^{C}(y|_{C}) .
$$
Combined this implies that $\rho(A_{F}^{C})(y|_{C})=(A_{F} y)|_{C} = A_{F}^{C} (y|_{C})$, and hence $y|_{C}=\lambda x_{F}^{C}|_{C}$ for some $\lambda \in [0, \infty[$. If $\overline{C} \setminus C= \emptyset$ it follows that $y=x_{F}^{C}$ since they both have unit $1$-norm, so assume $\overline{C} \setminus C \neq \emptyset$. Now $y-\lambda x_{F}^{C}$ is a vector supported on $\overline{C} \setminus C$ satisfying:
$$
A_{F} (y-\lambda x_{F}^{C})= \rho(A_{F}^{C}) (y-\lambda x_{F}^{C})
$$ 
but:
$$
\rho(A_{F}^{C}) (y-\lambda x_{F}^{C})|_{\overline{C} \setminus C}=\left(A_{F} (y-\lambda x_{F}^{C}) \right)|_{\overline{C} \setminus C}  = A_{F}^{\overline{C} \setminus C} (y-\lambda x_{F}^{C})|_{\overline{C} \setminus C}
$$
since $\rho(A_{F}^{\overline{C} \setminus C})<\rho(A_{F}^{C})$ this implies that $y=\lambda x_{F}^{C}$, and since they both have unit $1$-norm we must have that $y=x_{F}^{C}$.  
\end{proof}

\begin{lemma} \label{l67}
Let $F$ be a well chosen set for a finite $k$-graph $\Lambda$ without sources, and let $C$ be a $F$-harmonic component. The vector $x_{F}^{C}$ from Lemma \ref{lC} satisfies:
$$
A_{i}x_{F}^{C} = \rho(A_{i}^{C})x_{F}^{C}
$$
for $i=1, \dots ,k$.
\end{lemma}

\begin{proof}
Since $A_{i}^{\Lambda^{0}\setminus \overline{C}, \overline{C}}=0$ and $x_{F}^{C}|_{\Lambda^{0} \setminus \overline{C}}=0$ we have that:
$$
\left(A_{i}x_{F}^{C} \right)|_{\Lambda^{0} \setminus \overline{C}} = 0 .
$$
Now it follows that:
$$
A_{F} (A_{i} x_{F}^{C}) = A_{i}(A_{F}x_{F}^{C}) = \rho(A_{F}^{C})(A_{i}x_{F}^{C}) .
$$
Since $A_{i} x_{F}^{C} \in [0, \infty[^{\Lambda^{0}}$ Lemma \ref{lC} implies that $A_{i}x_{F}^{C}=\lambda_{i} x_{F}^{C}$ for some $\lambda_{i} \in [0, \infty[$ for each $i$, and hence:
\begin{equation*}
\lambda_{i} x_{F}^{C}|_{C} = \left (A_{i}x_{F}^{C} \right)|_{C} = A_{i}^{C} (x_{F}^{C}|_{C}) .
\end{equation*}
Since $x_{F}^{C}|_{C}$ is strictly positive we can conclude from Lemma 3.2 in \cite{aHLRS} that $\lambda_{i} = \rho(A_{i}^{C})$. By definition of $\lambda_{i}$ this proves the Lemma.
\end{proof}

Combining Lemma \ref{l67} and Lemma \ref{lC} it follows that a positive $F$-harmonic component $C$ gives rise to a $\beta$-harmonic vector for $\alpha^{r}$,  $x_{F}^{C}$, when $r$ is defined by:
$$
r: =\frac{1}{\beta} \left( \ln(\rho(A_{1}^{C})), \ln(\rho(A_{2}^{C})), \dots , \ln(\rho(A_{k}^{C})) \right) .
$$
We will now prove that all $\beta$-harmonic vectors for $\alpha^{r}$ can be decomposed as convex combinations of such vectors. To do this we will need the following technical Lemma. Notice that it deals with graphs that might have sources, which will prove important in its utilisation.

\begin{lemma}\label{lcomponent}
Let $\Lambda$ be a finite $k$-graph for some $k \in \mathbb{N}$, and let $B\in M_{\Lambda^{0}}([0, \infty[)$ be a matrix satisfying that for all $v,w \in \Lambda^{0}$ then $B(v,w)>0$ if and only if there exists a $n \in \mathbb{N}^{k}\setminus \{0\}$ with $v\Lambda^{n} w \neq \emptyset$. Then:
$$
\rho(B)=\max_{C \in \Lambda^{0}/\sim} \rho(B^{C}) .
$$
\end{lemma}
\begin{proof}
We will prove this by arranging the vertices of $\Lambda^{0}$ such that $B$ appears in a block upper triangular form with the block matrices consisting of the matrices $B^{C}$ with $C \in \Lambda^{0}/\sim$. This will prove the assertion in the Lemma, since the determinant of a block upper triangular matrix is the product of the determinants of the blocks.

To do this we define a directed $1$-graph $|\Lambda|=(V,E,r,s)$ by setting $V=\Lambda^{0}$, $E = \Lambda^{e_{1}} \cup \cdots \cup \Lambda^{e_{k}}$ and letting $s$ and $r$ be the restriction of the source and range map on $\Lambda$. Let $E^{*}$ denote the finite paths in $|\Lambda|$. Then using the factorisation property of $d$ it follows that for all $v,w \in V$ we have $vE^{*}w\neq \emptyset$ if and only if $v\Lambda w \neq \emptyset$. So defining a relation on $V$ by $v\leq w$ if $vE^{*}w \neq \emptyset$ we get exactly the same relation on $V=\Lambda^{0}$ as defined in Section \ref{s61}. Now we order the vertex set for the directed graph $|\Lambda |$ as it was done in Section 2.3 of \cite{aHLRS1}, giving us a numbering of the vertices $V= \{v_{1}, v_{2}, \dots, v_{|V|}\}$ satisfying that if $v_{i} \leq v_{j}$ then either $i\leq j$ or $v_{i} \sim v_{j}$, and that vertices in the same component are grouped together. For this order on $\Lambda^{0}$ it follows that $B$ has the desired form.
\end{proof}

The following proposition and its proof is inspired by Lemma 3.5 in \cite{CT}.

\begin{prop} \label{pdecomp}
Let $\Lambda$ be a finite $k$-graph without sources and let $\psi \in [0, \infty[^{\Lambda^{0}}$ be a $\beta$-harmonic vector for $\alpha^{r}$ for some $r \in \mathbb{R}^{k}$ and $\beta \in \mathbb{R}$. Let $F$ be well chosen. Then there exists a unique collection $\mathcal{C}$ of $F$-harmonic components and positive numbers $t_{C} \in ]0, 1]$ , $C \in \mathcal{C}$, such that:
$$
\psi = \sum_{C \in \mathcal{C}} t_{C} x_{F}^{C} .
$$
Furthermore each $C \in \mathcal{C}$ is positive with $C \nleq C'$ for $C' \in \mathcal{C}\setminus \{C\}$, and each $C \in \mathcal{C}$ satisfies:
$$
\beta r =(\ln(\rho(A_{1}^{C})), \dots, \ln(\rho(A_{k}^{C})))
$$
and that $x_{F}^{C}$ is a $\beta$-harmonic vector for $\alpha^{r}$.
\end{prop}

\begin{proof}
We will first prove that such a decomposition exists. Since $\psi$ is $\beta$-harmonic for $\alpha^{r}$ we know that $ A_{i}\psi = e^{\beta r_{i}} \psi $ for all $ i=1,2, \dots , k$. This implies that:
$$
A_{F}\psi = \sum_{n \in F}A^{n} \psi = \sum_{n \in F}e^{\beta r \cdot n} \psi .
$$
Let $K: = \sum_{n \in F}e^{\beta r \cdot n}$, then $K>0$ and we let 
$$
B:=K^{-1}A_{F} \in M_{\Lambda^{0}}([0,\infty[) .
$$
$B$ is then a non-negative matrix with the property that $B_{v,w}>0$ if and only if there is a non-trivial path from $w$ to $v$, and:
$$
B\psi = \psi .
$$
Set $W := \{v \in \Lambda^{0} \ | \ \psi_{v}>0  \}$. We claim $W$ is closed, i.e. $\overline{W}=W$. To see this, let $\lambda \in v \Lambda w$ for some $w \in W$ and $v \neq w$, then $B_{v,w}>0$ and hence:
$$
0 < B_{v,w}\psi_{w} \leq (B\psi)_{v} = \psi_{v} \Rightarrow v \in W .
$$
For any $v,w \in W$ we have $B^{n}_{v,w} \psi_{w} \leq (B^{n}\psi)_{v}\leq\psi_{v}$, so setting 
$$
L:= \max\{ \psi_{v}/\psi_{w} \ | \ v,w \in W\} >0
$$ 
we get that $B^{n}_{v,w} \leq L$ for all $n \in \mathbb{N}$ and $v,w \in W$. Using the properties of $B$ and $W$ we get that $(B^{W})^{n}_{v,w}=B^{n}_{v,w}$ for all $v,w \in W$, so $\lVert (B^{W})^{n} \rVert_{F} \leq |W| \cdot L$ for all $n\in \mathbb{N}$ where $\lVert \cdot \rVert_{F}$ denotes the Frobenius norm. By Gelfands Formula:
$$
\rho(B^{W})=\lim_{n \to \infty}\lVert (B^{W})^{n} \rVert_{F}^{1/n} \leq 1 .
$$
Since $B\psi =\psi$ we get that $\psi|_{W} = (B\psi)|_{W}=B^{W}(\psi|_{W})$, and hence $\rho(B^{W})=1$. Using Lemma \ref{lcomponent} on the graph $\Lambda_{W}$ we get:
$$
\rho(B^{W}) = \max_{C } \rho(B^{C})
$$
where $\max$ is taking over the components $C$ in $\Lambda_{W}$. Let $\mathcal{C}'$ be the collection of components $C$ in $W$ with $\rho(B^{C})=1$, and let $\mathcal{C}$ be the minimal elements of $\mathcal{C}'$ with respect to the order $\leq$. We now claim that $\mathcal{C}$ consists of $F$-harmonic components $C$ satisfying $\rho(A_{F}^{C})=K$. For $C \in \mathcal{C}$ we have:
$$
1=\rho(B^{C}) = \rho((K^{-1}A_{F})^{C}) = \rho(A_{F}^{C})/K \Rightarrow  \rho(A_{F}^{C})=K .
$$
Since $K>0$ this also implies that $C$ is non-trivial. Since $\overline{C} \subseteq W$ we have as before that $\rho(B^{\overline{C} \setminus C})\leq 1$ using Gelfand. If $\rho(B^{\overline{C} \setminus C})= 1$ there must be some component $D \subseteq \overline{C} \setminus C$ with $\rho(B^{D})=1$, but since this implies that $D \leq C$, $D \neq C$ and $D \in \mathcal{C}'$ this cannot be the case. So $\rho(B^{\overline{C} \setminus C})<1$ and hence:
$$
1 > \rho(B^{\overline{C} \setminus C}) = \rho((K^{-1}A_{F})^{\overline{C} \setminus C}) = \rho(A_{F}^{\overline{C} \setminus C})/K \Rightarrow \rho(A_{F}^{\overline{C} \setminus C}) <K=\rho(A_{F}^{C})
$$
proving that $C$ is in fact $F$-harmonic.

For any $D \in \mathcal{C}$ we have that $B^{D}(\psi|_{D}) \leq (B\psi)|_{D}=\psi|_{D}$, so:
$$
A_{F}^{D}(\psi|_{D}) \leq K\psi|_{D} = \rho(A_{F}^{D})\psi|_{D} .
$$
Since $\psi|_{D}$ is strictly positive the subinvariance theorem now imply that $A_{F}^{D}(\psi|_{D})  = \rho(A_{F}^{D})\psi|_{D}$, and hence $\psi|_{D}$ is a positive eigenvector for $A_{F}^{D}$ with eigenvalue $\rho(A_{F}^{D})$. However this is also the case for $x^{D}_{F}|_{D}$, so there is a positive number $t_{D}>0$ such that $\psi|_{D} = t_{D}x^{D}_{F}|_{D}$. Set:
$$
\eta = \psi - \sum_{D \in \mathcal{C}}t_{D} x^{D}_{F} .
$$
Since $A_{F} x^{D}_{F} = \rho(A_{F}^{D}) x^{D}_{F}=Kx_{F}^{D}$ we see that $B x^{D}_{F}=x^{D}_{F}$ and hence $B\eta=\eta$. The vector $\eta \in \mathbb{R}^{\Lambda^{0}}$ has $\eta_{v}=0$ for $v \notin W$. By definition $D, D' \in \mathcal{C}$ has $D \cap \overline{D'} \neq \emptyset$ if and only if $D=D'$, so we also have that $\eta_{v}=0$ for $v \in J:=\bigcup_{D\in \mathcal{C}} D$. Set $H$ to be the hereditary closure of $J$ in $\Lambda$, $H=\widehat{J}$, and consider $D \in \mathcal{C}$. If $v \in (H \setminus J) \cap \overline{D}$, then there is $D' \in \mathcal{C}$ such that $D' \Lambda v \neq \emptyset$ and $v \Lambda D \neq \emptyset$, and hence by composing paths $D' \Lambda D \neq \emptyset$. So $D' \leq D$. If $D' = D$ then $v \in D \subseteq J$, so $D \neq D'$, and since $\mathcal{C}$ consists of minimal elements we reach a contradiction. So $(H \setminus J) \cap \overline{D} = \emptyset$ and hence $\eta |_{H \setminus J} =\psi |_{H \setminus J}\geq 0$. For any $w \in H\setminus J$ we have a $v \in J$ such that $v \Lambda w \neq \emptyset$, i.e. $B_{v,w} \neq 0$, and hence:
$$
0=\eta_{v} = (B\eta)_{v} = \sum_{u \in \Lambda^{0}} B_{v,u} \eta_{u} = \sum_{u \in H \setminus J} B_{v,u} \eta_{u} \geq B_{v,w} \eta_{w} \geq 0 .
$$
So $\eta|_{H}=0$. Since $H$ contains all components $C$ in $W$ with $\rho(B^{C})=1$, we get that $\rho(B^{W \setminus H}) <1$. However for $v \in W \setminus H$ we have:
$$
(B^{W \setminus H} \eta|_{W\setminus H})_{v} = \sum_{w \in W \setminus H} B_{v,w} \eta_{w} = \sum_{w \in W} B_{v,w} \eta_{w} = \sum_{w \in \Lambda^{0}} B_{v,w} \eta_{w}=\eta_{v} .
$$
So $\eta|_{W\setminus H}=0$, and hence $\eta =0$, proving existence.

To prove uniqueness, assume that $\mathcal{D}$ is a collection of $F$-harmonic components and that there exists $s_{D}>0$ for all $D \in \mathcal{D}$ such that:
$$
\psi= \sum_{D \in \mathcal{D}} s_{D} x^{D}_{F} .
$$
So $W = \bigcup_{D \in \mathcal{D}} \overline{D} = \bigcup_{C \in \mathcal{C}} \overline{C}$ and hence for any $C \in \mathcal{C}$ there is a $D \in \mathcal{D}$ with $C \subseteq \overline{D}$. Assume $C \neq D$, then there is another $C' \in \mathcal{C}$ such that $D \subseteq \overline{C'}$, and hence $C \subseteq \overline{C'}$ with $C \neq C'$, contradicting the choice of $\mathcal{C}$. So $\mathcal{C} \subseteq \mathcal{D}$ and for $C \in \mathcal{C}$ the only $D \in \mathcal{D}$ with $\overline{D} \cap C \neq \emptyset$ is $D =C$, and hence we get $\psi|_{C}=s_{C}x_{F}^{C}|_{C}$. By choice of $\mathcal{C}$ we know $\overline{C'} \cap C = \emptyset$ for $C' \in \mathcal{C}\setminus\{C\}$ and hence we also have $\psi|_{C}=t_{C}x_{F}^{C}|_{C}$.  This implies that $t_{C}=s_{C}$ for $C \in \mathcal{C}$, and hence:
$$
0=\psi-\psi = \sum_{D \in \mathcal{D}} s_{D} x_{F}^{D} - \sum_{C \in \mathcal{C}} t_{C} x_{F}^{C} =\sum_{D \in \mathcal{D}\setminus \mathcal{C}} s_{D} x_{F}^{D}
$$
proving the uniqueness. Let $i \in \{1,2, \dots, k\}$. Lemma \ref{l67} implies that each $x^{C}_{F}$ satisfies:
$$
A_{i}x_{F}^{C}= \rho(A_{i}^{C}) x_{F}^{C} .
$$
By choice of $\psi$, $A_{i}\psi = e^{\beta r_{i}} \psi$, so multiplying with $A_{i}e^{-\beta r_{i}}$ gives:
$$
\psi=\sum_{C \in \mathcal{C}} t_{C} \rho(A_{i}^{C}) e^{-\beta r_{i}}   x_{F}^{C}
$$
we now use the uniqueness result to get that:
$$
\rho(A_{i}^{C}) e^{-\beta r_{i}}=1 \Rightarrow \rho(A_{i}^{C}) = e^{\beta r_{i}}
$$
for all $C \in \mathcal{C}$. This proves the last statements.
\end{proof}

\begin{cor}\label{cinde}
To be a positive $F$-harmonic component is independent of choice of well chosen $F$, and the vectors $x_{F}^{C}$ are independent of choice of $F$.
\end{cor}
\begin{proof}
Assume that $C$ is a positive $F$-harmonic component for some well chosen $F$, then there is a vector $x_{F}^{C}$ such that $(x_{F}^{C})_{v}=0$ for $v \notin \overline{C}$ and that is $1$-harmonic for $\alpha^{r}$ where $r=(\ln(\rho(A_{1}^{C})), \dots, \ln(\rho(A_{k}^{C})))$. Let $\tilde{F}$ be another well chosen set, then by Proposition \ref{pdecomp} there is a collection $\mathcal{D}$ of $\tilde{F}$-harmonic components with:
$$
x_{F}^{C} = \sum_{D \in \mathcal{D}} t_{D} x_{\tilde{F}}^{D} .
$$
This implies that $\overline{C} = \bigcup_{D \in \mathcal{D}} \overline{D}$, so $C \in \mathcal{D}$ and hence $C$ is a $\tilde{F}$-harmonic component. If there were a $D' \in \mathcal{D}$ with $D' \neq C$, then $D' \subseteq\overline{C} \setminus C$ which is impossible by choice of $\mathcal{D}$. So $\mathcal{D}=\{C\}$ and since $x_{F}^{C}$ and $x_{\tilde{F}}^{C}$ have unit $1$-norm $x_{F}^{C}=x_{\tilde{F}}^{C}$.
\end{proof}

Corollary \ref{cinde} justifies that we drop the $F$ and simply call it a \emph{positive harmonic component}, and denote the vectors $x^{C}$. When we in the following write $r^{1} \lneq r^{2}$ for vectors $r^{1}, r^{2} \in \mathbb{R}^{k}$ we mean that $r^{1}_{i} \leq r^{2}_{i}$ for all $i$, but that $r^{1} \neq r^{2}$.

\begin{lemma} \label{lhar}
$C$ is a positive harmonic component if and only if it is positive and
\begin{equation}\label{elemma}
\left (\rho(A_{1}^{D}) ,  \rho(A_{2}^{D}) , \dots ,  \rho(A_{k}^{D}) \right ) \lneq \left (\rho(A_{1}^{C}) ,  \rho(A_{2}^{C}) , \dots ,  \rho(A_{k}^{C}) \right )
\end{equation}
for all components $D \subseteq \overline{C} \setminus C$. 
\end{lemma}

\begin{proof}
We will first argue that for all components $D$ and well chosen $F$ we have:
\begin{equation}\label{eeq}
\rho(A_{F}^{D}) = \sum_{n \in F} \prod_{i=1}^{k} \rho(A_{i}^{D})^{n_{i}} .
\end{equation}
If $D$ is trivial this is true since both sides equal $0$, so assume $D$ is non-trivial. Then $A_{F}^{D}$ is strictly positive and hence has a unimodular Perron-Frobenious eigenvector $z$, and since $A_{i}A_{F}=A_{F}A_{i}$ it follows that $A_{F}^{D}$ and $A_{i}^{D}$ commute, so $A_{F}^{D}A^{D}_{i}z =\rho(A_{F}^{D})A^{D}_{i}z$. Hence $A^{D}_{i}z=\lambda_{i}z$ with $\lambda_{i} \geq 0$ for all $i$. Lemma 3.2 in \cite{aHLRS} then implies that $\lambda_{i}=\rho(A^{D}_{i})$, and hence:
$$
\rho(A_{F}^{D})z = A_{F}^{D}z =\big(\sum_{n \in F} \prod_{i=1}^{k} A_{i}^{n_{i}} \big)^{D}z=\sum_{n \in F}  \big(\prod_{i=1}^{k} A_{i}^{n_{i}} \big)^{D}z .
$$
So if we can argue that $(\prod_{i} A_{i}^{n_{i}} )^{D}= \prod_{i} (A_{i}^{D})^{n_{i}}$ this proves \eqref{eeq}. This equality follows from a straightforward induction argument on $l=n_{1}+\cdots n_{k}$ and the fact that if $B_{1}, \dots ,B_{r}$ is a set of non-negative matrices over $\Lambda^{0}$ satisfying that $B_{i}(v,w)>0$ implies $v\Lambda w \neq \emptyset$, then $B_{1}B_{2}\cdots B_{r}$ has the same property.

Assume that $C$ is a positive harmonic component, then for any well chosen $F$:
$$
\rho(A_{F}^{\overline{C} \setminus C})=\max_{D \subseteq \overline{C} \setminus C} \rho(A_{F}^{D})
$$
so $\rho(A_{F}^{D}) < \rho(A_{F}^{C})$ for every component $D \subseteq \overline{C} \setminus C$ no matter the choice of well chosen $F$. Now fix a well chosen $F=\{a_{1}, \dots , a_{m}\}$. Assume for contradiction that there is a component $D \subseteq \overline{C} \setminus C$ and a $j$ with $\rho(A_{j}^{D}) > \rho(A_{j}^{C})$, and notice that this implies that $D$ is non-trivial. Define for $ s,l \in \mathbb{N}$, $F_{s} = \{a_{1}+se_{j} , \dots , a_{m}+se_{j}\}$ and:
$$
F_{l,s}: = F \cup \bigcup_{i=1}^{l} F_{s} .
$$
Then $F_{l,s}$ is well chosen, and hence using \eqref{eeq} we get that:
$$
\rho(A_{F}^{D})+\rho(A_{j}^{D})^{s}\rho(A_{F}^{D})\cdot l =\rho(A_{F_{l,s}}^{D})  \ < \  \rho(A_{F_{l,s}}^{C}) = \rho(A_{F}^{C})+\rho(A_{j}^{C})^{s}\rho(A_{F}^{C})\cdot l
$$
for all $l, s \in \mathbb{N}$.  This implies that:
$$
\rho(A_{F}^{D}) \leq \frac{1+\rho(A_{j}^{C})^{s} \cdot l}{1+\rho(A_{j}^{D})^{s} \cdot l} \rho(A_{F}^{C})
$$
for all $l,s$ and hence letting $l \to \infty$ we get that for all $s$:
$$
\rho(A_{F}^{D}) \leq \frac{\rho(A_{j}^{C})^{s}}{\rho(A_{j}^{D})^{s}} \rho(A_{F}^{C})
$$
and hence letting $s \to \infty$ and using $\rho(A_{j}^{D}) > \rho(A_{j}^{C})$ we get that $\rho(A_{F}^{D}) = 0$, in contradiction to the fact that $A_{F}^{D}$ is a strictly positive integer matrix. If there were a $D \subseteq \overline{C} \setminus C$ with $\rho(A_{i}^{D}) = \rho(A_{i}^{C})$ for each $i$ then \eqref{eeq} would imply that $\rho(A_{F}^{D}) = \rho(A_{F}^{C})$, also a contradiction.

Assume on the other hand that $C$ is positive and satisfies \eqref{elemma}, then \eqref{eeq} implies that $\rho(A_{F}^{D}) \leq \rho(A_{F}^{C})$ for any $D \subseteq \overline{C} \setminus C$ and well chosen $F$. Fix a well chosen $F=\{a_{1}, \dots , a_{m} \}$ and define $F_{i}=\{a_{1}+e_{i}, \dots, a_{m}+e_{i} \}$ for $i=1, \dots , k$ and $\tilde{F}: = F \cup F_{1} \cup \cdots \cup F_{k}$. By definition of $\tilde{F}$, any component $D \subseteq \overline{C} \setminus C$ satisfies:
$$
\rho(A_{\tilde{F}}^{D}) =\rho(A_{F}^{D}) +\sum_{i=1}^{k}\rho(A_{i}^{D}) \rho(A_{F}^{D}) <\rho(A_{F}^{C}) +\sum_{i=1}^{k}\rho(A_{i}^{C}) \rho(A_{F}^{C}) = \rho(A_{\tilde{F}}^{C}) .
$$
Since this is true for all $D \subseteq \overline{C} \setminus C$ we get that $\rho(A_{\tilde{F}}^{\overline{C}\setminus C}) < \rho(A_{\tilde{F}}^{C})$ and hence $C$ is a positive $\tilde{F}$-harmonic component.
\end{proof}

Fix some $r \in \mathbb{R}^{k}$. For each $\beta \in \mathbb{R}$ we set $\mathcal{C}_{r}(\beta)$ to be the set of positive harmonic components $C$ satisfying that $\beta r=(\ln(\rho(A_{1}^{C})), \dots ,\ln(\rho(A_{k}^{C}))) $.

\begin{thm}\label{t712}
Let $\Lambda$ be a finite $k$-graph without sources and let $r \in \mathbb{R}^{k}$ and $\beta \in \mathbb{R}$. There is an affine bijective correspondence between the gauge-invariant $\beta$-KMS states $\omega$ for $\alpha^{r}$ and the functions $f:\mathcal{C}_{r}(\beta) \to [0,1]$ with $\sum_{C \in \mathcal{C}_{r}(\beta)} f(C)=1$. A state $\omega$ corresponding to a function $f$ is given by:
$$
\omega (t_{\lambda} t_{\gamma}^{*}) = \delta_{\lambda, \gamma} e^{-\beta r\cdot d(\lambda)} \psi_{s(\lambda)}
$$
for all $\lambda , \gamma \in \Lambda$, where $\psi \in [0, \infty[^{\Lambda^{0}}$ is given by:
$$
\psi = \sum_{C \in \mathcal{C}_{r}(\beta)} f(C) x^{C}.
$$
\end{thm}
\begin{proof}
Let $\omega $ be a gauge-invariant $\beta$-KMS state for $\alpha^{r}$ and $\psi$ be the corresponding unique $\beta$-harmonic vector for $\alpha^{r}$ given by Theorem \ref{tvector}, then for any $\lambda , \gamma \in \Lambda$:
$$
\omega(t_{\lambda}t_{\gamma}^{*})=\int_{\Lambda^{\infty}} P(t_{\lambda}t_{\gamma}^{*}) \ dM_{\psi}  = \delta_{\lambda, \gamma} M_{\psi} (Z(\lambda))
= \delta_{\lambda, \gamma} e^{-\beta r \cdot d(\lambda)} \psi_{s(\lambda)} .
$$
That it is an affine bijection follows from Proposition \ref{pdecomp} and the definition of $\mathcal{C}_{r}(\beta)$.
\end{proof}

\section{The non gauge-invariant KMS states}

We will now use Theorem \ref{t712} and the symmetries of the KMS-simplex to obtain a description of all the KMS states. The map $\psi \to M_{\psi}$ is an affine bijection from the $\beta$-harmonic vectors for $\alpha^{r}$ to the set of $e^{-\beta c_{r}}$-quasi-invariant measures, where $c_{r}(x,a,y)=a\cdot r$. So the extreme points of the simplex $\tilde{\Delta}$ of $e^{-\beta c_{r}}$-quasi-invariant probability measures are the measures $M_{C}:=M_{x^{C}}$, where $C \in \mathcal{C}_{r}(\beta)$. To use Theorem \ref{tmain} we first have to analyse the paths in $\Lambda^{\infty}$. For a subset $S \subseteq \Lambda^{0}$ we say that a path $x \in \Lambda^{\infty}$ \emph{eventually lies in $S$} if there exists a $n \in \mathbb{N}^{k}$ such that:
$$
r(\sigma^{m}(x)) \in S \qquad \forall m \geq n .
$$
This concept proves important for describing the measures $M_{C}$, $C \in \mathcal{C}_{r}(\beta)$.

\begin{lemma}
Let $\Lambda$ be a finite $k$-graph without sources. For any component $D$ the set:
$$
N_{D}:=\{x \in \Lambda^{\infty} \ | \    x \text{ eventually lies in } D \}
$$
is a Borel set in $\Lambda^{\infty}$. For any $r \in \mathbb{R}^{k}$, $\beta \in \mathbb{R}$ and $D \in \mathcal{C}_{r}(\beta)$ we have $M_{D}(N_{D})=1$.
\end{lemma}
\begin{proof}
We first want to argue that:
$$
N_{\overline{C}}:=\{x \in \Lambda^{\infty} \ | \    x \text{ eventually lies in } \overline{C} \}
$$
is a closed set for all components $C$. So let $y \in \Lambda^{\infty}\setminus N_{\overline{C}}$. Then there is a $m \in \mathbb{N}^{k}$ such that $r(\sigma^{m}(y)) \notin \overline{C}$, and we set $\lambda: =y(0,m)\in \Lambda $. We claim that $Z(\lambda) \cap N_{\overline{C}} = \emptyset$. To see this, assume $z \in Z(\lambda) \cap N_{\overline{C}}$, then there exists a $N\geq m$ such that $r(\sigma^{N}(z)) \in \overline{C}$, however then $z(m ,N )$ is a path with $r(z(m ,N ))=s(\lambda) = r(\sigma^{m}(y))$ and $s(z(m ,N ))=r((\sigma^{N}(z)) \in \overline{C}$, in contradiction to the fact that $r(\sigma^{m}(y)) \notin \overline{C}$. To prove that $N_{D}$ is Borel, it suffices to prove that:
$$
N_{D} = N_{\overline{D}} \setminus \left( \bigcup_{\text{components } C \subseteq \overline{D} \setminus D} N_{\overline{C}}  \right) .
$$
If $N_{D} \cap N_{\overline{C}} \neq \emptyset$ for some $C \subseteq \overline{D}$, we must have $\overline{C} \cap D \neq \emptyset$, which implies that $D=C$. This proves $"\subseteq"$. For a path $z$ in the right hand side, numerate the finite collection of components $C_{1}, C_{2}, \dots , C_{l} \subseteq \overline{D} \setminus D$, and let $N_{1}, \dots , N_{l} \in \mathbb{N}^{k}$ be numbers such that $r(\sigma^{N_{i}}(z)) \notin \overline{C_{i}}$. It then follows that $r(\sigma^{N}(z)) \notin \bigcup_{i=1}^{l} \overline{C_{i}}$ for all $N\geq  N_{1}\vee \cdots \vee N_{l}$. There is a $N\geq  N_{1}\vee \cdots \vee N_{l}$ such that $r(\sigma^{m}(z))\in \overline{D}$ for all $m \geq N$, so $r(\sigma^{m}(z))\in D$ for all $m \geq N$, and hence $z \in N_{D}$.

Let $r \in \mathbb{R}^{k}$, $\beta \in \mathbb{R}$ and $D \in \mathcal{C}_{r}(\beta)$. We will first prove that $M_{D}(N_{\overline{D}})=1$. It is enough to prove that $M_{D}(Z(\lambda))=0$ for $\lambda \in \Lambda$ with $s(\lambda) \notin \overline{D}$. By definition:
$$
M_{D}(Z(\lambda)) = e^{-\beta r \cdot d(\lambda)}x^{D}_{s(\lambda)}
$$
however $x^{D}$ is supported on $\overline{D}$, so $M_{D}(Z(\lambda))=0$, proving that $M_{D}(N_{\overline{D}})=1$. Now assume for contradiction that $M_{D}(N_{\overline{C}}) \neq 0$ for a $C \subseteq \overline{D} \setminus D$. If $M_{D}(N_{\overline{C}})=1$ then $M_{D}(Z(v))=0$ for $v \in D$, since then $Z(v) \cap N_{\overline{C}}=\emptyset$, however $M_{D}(Z(v))=x^{D}_{v} >0$, so $M_{D}(N_{\overline{C}}) \in ]0,1[$. Notice that clearly $N_{\overline{C}}$ is invariant in the sense that $r(s^{-1}(N_{\overline{C}})) = s(r^{-1}(N_{\overline{C}} ))= N_{\overline{C}}$, and hence as noted earlier we can decompose $M_{D}$ as a non-trivial convex combination of two $e^{-\beta c_{r}}$-quasi-invariant measures, contradicting that $M_{D}$ is extremal. So $M_{D}(N_{\overline{C}})=0$ which proves $M_{D}(N_{D})=1$.
\end{proof}

Given a component $D \in \mathcal{C}_{r}(\beta)$ consider the graph $\Lambda_{D}$ which is a strongly connected $k$-graph. Hence as in \cite{aHLRS} it has a Periodicity-group $\text{Per}(\Lambda_{D}) \subseteq \mathbb{Z}^{k}$ associated with it. We denote this subgroup of $\mathbb{Z}^{k}$ as $\text{Per}(D):=\text{Per}(\Lambda_{D})$, and remind the reader that:
$$
\text{Per}(D) = \{ m-n \ | \ m,n \in \mathbb{N}^{k} , \ \sigma^{m}(x)=\sigma^{n}(x) \text{ for all } x \in \Lambda_{D}^{\infty} \}
$$ 
c.f. Proposition 5.2 in \cite{aHLRS}. We let $\Phi:\mathcal{G} \to \mathbb{Z}^{k}$ denote the map $(x,a,y)\to a$. We can now obtain the entire description of KMS states. When $\Lambda$ is strongly connected this description follows from Theorem 7.1 in \cite{aHLRS}.

\begin{thm} \label{t82}
Let $\Lambda$ be a finite $k$-graph without sources and $r \in \mathbb{R}^{k}$ and $\beta \in \mathbb{R}$. There is a bijection from the pairs $(C, \xi)$ consisting of a $C \in \mathcal{C}_{r}(\beta)$ and a $\xi \in \widehat{\text{Per}(C)}$ to the set of extremal $\beta$-KMS states for $\alpha^{r}$ given by:
$$
(C, \xi ) \to \omega_{C, \xi}
$$
where:
$$
\omega_{C, \xi}(f) = \int_{X(\text{Per}(C))} \sum_{g \in \mathcal{G}_{x}^{x}} f(g) \xi(\Phi(g)) \ d M_{C}(x) \qquad \forall f \in C_{c}(\mathcal{G}) .
$$
\end{thm}
\begin{proof}
To use Theorem \ref{tmain} we assume for now that $\beta \neq 0$. We will prove that for a $D \in \mathcal{C}_{r}(\beta)$, the unique subgroup of $\mathbb{Z}^{k}$ described in $(1)$ in Theorem \ref{tmain} for the measure $M_{D}$ is $\text{Per}(D)$. Assume that $y \in N_{D}$, then there is a $m\in \mathbb{N}^{k}$ such that $r(\sigma^{l}(y)) \in D$ for all $l \geq m$, and hence $\sigma^{m}(y)$ can be considered as an infinite path in the graph $\Lambda_{D}$. It follows that for $n_{1}-n_{2} \in \text{Per}(D)$ we have:
$$
\sigma^{n_{1}+m}(y) = \sigma^{n_{1}}(\sigma^{m}(y)) =  \sigma^{n_{2}}(\sigma^{m}(y)) = \sigma^{n_{2}+m}(y)
$$
hence $(y,n_{1}-n_{2},y)\in \mathcal{G}_{y}^{y}$. So:
$$
1=M_{D}(N_{D}) = M_{D}(\{x \in \Lambda^{\infty} \ | \ \{x\}\times\text{Per}(D) \times \{x\}\subseteq \mathcal{G}_{x}^{x}  \}) 
$$
and hence the subgroup $B$ from Theorem \ref{tmain} satisfies $\text{Per}(D) \subseteq B$. Assume for a contradiction that $\text{Per}(D) \subsetneq B$, then for $l \in B \setminus \text{Per}(D)$ we have: 
$$
M_{D}(\{x \in N_{D}\ |\ (x,l,x) \in \mathcal{G}_{x}^{x}\})=1 .
$$
Since $M_{D}(Z(v))=x^{D}_{v}>0$ for a $v \in D$ and since:
$$
\{x \in vN_{D} \ |\ (x,l,x) \in \mathcal{G}_{x}^{x}\} \subseteq \bigcup_{n,m \in \mathbb{N}^{k}, \ n-m=l} \{x \in vN_{D}\ |\ \sigma^{n}(x)=\sigma^{m}(x)\}
$$
there must be a $n_{1}, n_{2} \in \mathbb{N}^{k}$ with $M_{D}(\{x \in vN_{D} \ |\ \sigma^{n_{1}}(x)=\sigma^{n_{2}}(x)\})>0$ and $n_{1}-n_{2}=l$. Now consider the measure $M$ defined on the strongly connected graph $\Lambda_{D}$ as in \cite{aHLRS}, since $l \notin \text{Per}(D)$ we have by Proposition 8.2 in \cite{aHLRS} that:
$$
M(\{ x \in v\Lambda_{D}^{\infty} \ | \ \sigma^{n_{1}}(x)=\sigma^{n_{2}}(x)\})=0 .
$$
Since this set is compact we can choose an arbitrary $\varepsilon >0$ and find a finite number of paths $\delta_{i} \in \Lambda_{D}$, $i=1, \dots, n$ such that letting $Z_{D}(\delta_{i})=\{x \in \Lambda_{D}^{\infty} \ | \ x(0, d(\delta_{i}))=\delta_{i}\}$ for each $i$ we have $Z_{D}(\delta_{i})\cap Z_{D}(\delta_{j})=\emptyset$ for $i \neq j$ and:
$$
\{ x \in v\Lambda_{D}^{\infty} \ | \ \sigma^{n_{1}}(x)=\sigma^{n_{2}}(x)\} \subseteq \bigcup_{i=1}^{n} Z_{D}(\delta_{i}) \quad , \quad \sum_{i=1}^{n} M(Z_{D}(\delta_{i}))\leq \varepsilon .
$$
The paths $\delta_{i} \in \Lambda_{D}$ can be considered as paths in $\Lambda$, and hence denoting by $Z(\delta_{i})=\{x \in \Lambda^{\infty} \ | \ x(0, d(\delta_{i}))=\delta_{i}\}$ it is straightforward to check that:
$$
\{x \in vN_{D} \ |\ \sigma^{n_{1}}(x)=\sigma^{n_{2}}(x)\} \subseteq \bigcup_{i=1}^{n} Z(\delta_{i}) .
$$ 
By definition of $M_{D}$ and $x^{D}$ there is a $c\in ]0,1]$ such that $x^{D}|_{D}= cx$ where $x$ is the unimodular Perron-Frobenious eigenvector for $A_{F}^{D}$. Since $A_{i}x^{D}=\rho(A_{i}^{D})x^{D}$ it follows that $A_{i}^{D}(x^{D}|_{D})=\rho(A_{i}^{D})(x^{D}|_{D})$, so since $A_{1}^{D}, \dots, A_{k}^{D}$ are the vertex matrices for $\Lambda_{D}$, it follows that $x$ is the unimodular Perron-Frobenius eigenvector of $\Lambda_{D}$, c.f. Definition 4.4 in \cite{aHLRS}. So by definition of $M$ in Section 8 of \cite{aHLRS} we get:
\begin{align*}
M_{D}\left(\bigcup_{i=1}^{n} Z(\delta_{i})\right ) &\leq \sum_{i=1}^{n} M_{D}(Z(\delta_{i})) =  \sum_{i=1}^{n} e^{-\beta r \cdot d(\delta_{i})} x^{D}_{s(\delta_{i})} = c \sum_{i=1}^{n} e^{-\beta r \cdot d(\delta_{i})} x_{s(\delta_{i})} \\
&= c \sum_{i=1}^{n} M(Z_{D}(\delta_{i}))\leq c \varepsilon \leq \varepsilon
\end{align*}
since $\varepsilon$ was arbitrary, we reach our contradiction, and hence $B = \text{Per}(D)$. In the case where $\beta=0$ we notice that the $\beta$-KMS states for $\alpha^{r}$ are the same as the $1$-KMS states for $\alpha^{0}$, with $0 \in \mathbb{R}^{k}$, since $\alpha^{0}$ is the trivial one-parameter group. However $\mathcal{C}_{r}(0)=\mathcal{C}_{0}(1)$, so we also have a bijection in this case.
\end{proof}

\begin{remark}
In our setting the $C^{*}$-algebra $C^{*}(\Lambda)$ is simple if and only if $\Lambda$ is cofinal and has no local periodicity, c.f.  Theorem 3.1 in \cite{RS}. Since $\Lambda$ has no sources it has to contain some positive component $C$, and since it is cofinal $C$ has to be the only positive component and it has to satisfy $\hat{C}=C$. Since $\Lambda$ has no local periodicity it follows that $\text{Per}(C)$ is trivial. To see that $C$ is a harmonic component, assume $D$ is another component and $i \in \{1,2, \dots, k\}$. Since $C$ is the only positive component, there exists a $n\in \mathbb{N}^{k}$ such that $\Lambda^{n}=\Lambda^{n}C$. Setting $N:=|D\Lambda^{n}|$ and letting $l \in \mathbb{N}$ be arbitrary, it follows for each $v,w \in D$ and fixed $\gamma \in w\Lambda^{n}$ that the map: 
$$
v\Lambda^{le_{i}} w \ni \lambda \xrightarrow{\varphi} (\lambda \gamma)(n, n+le_{i}) \in C\Lambda^{le_{i}}C
$$
has at most $N$ points in $\varphi^{-1}(\{\nu\})$ for each $\nu \in C\Lambda^{le_{i}}C$, so $|v\Lambda^{le_{i}} w| \leq N |\varphi \left( v\Lambda^{le_{i}} w  \right)| \leq N |C|^{2} \cdot \lVert (A_{i}^{C})^{l} \rVert_{\max}$. It follows that:
$$
\lVert (A_{i}^{D})^{l}\rVert_{F} \leq |\Lambda^{0}| \lVert (A_{i}^{D})^{l}\rVert_{\max} \leq  |\Lambda^{0}| N |C|^{2} \cdot \lVert (A_{i}^{C})^{l} \rVert_{F}
$$
By Gelfand's formula $\rho(A_{i}^{D}) \leq \rho(A_{i}^{C})$,  so since $D$ is not positive we conclude that $C$ is harmonic, and Theorem \ref{t82} then implies that there is exactly one $\beta$-KMS state for $\alpha^{r}$ if $r=\frac{1}{\beta} (\ln(A^{C}_{1}), \dots , \ln(A^{C}_{k}))$ and no $\beta$-KMS states for $\alpha^{r}$ for any other choices of $\beta$.
\end{remark}

\section{Acknowledgement}
The bulk of this work was done while visiting Astrid an Huef and Iain Raeburn at the University of Otago for a longer period of time, and the author is immensely grateful for the enlightening discussions and for the great hospitality shown to him by the entire O.A. group. This stay was primarily financed by the grant \emph{6161-00012B Eliteforsk legat} from the Danish Ministry of Higher Education and Science. The author also thanks Jean Renault for sharing insight that led to an improvement of Definition \ref{d2}, and lastly the author thanks Klaus Thomsen for supervision.


\begin{thebibliography}{WWWW} 


\bibitem{BR} O. Bratteli, D.W. Robinson, {\em Operator Algebras
    and Quantum Statistical Mechanics I + II}, Texts and Monographs in
  Physics, Springer Verlag, New York, Heidelberg, Berlin, 1979 and 1981.
  
  

\bibitem{CT} J. Christensen and K. Thomsen, {\em Finite digraphs and KMS states}, J. Math. Anal. Appl. {\bf 433} (2016), 1626-1646. https://doi.org/10.1016/j.jmaa.2015.08.060.

\bibitem{ER} R. Exel , J. Renault, {\em Semigroups of local homeomorphisms and interaction groups }, Ergodic Theory and Dynamical Systems {\bf 27(6)} (2007), 1737-1771. https://doi.org/10.1017/S0143385707000193.



\bibitem{aHKR} A. an Huef, S. Kang, I. Raeburn, {\em KMS States on the Operator Algebras of Reducible Higher-Rank Graphs}, Integr. Equ. Oper. Theory {\bf 88} (2017), 91 -126.  https://doi.org/10.1007/s00020-017-2356-z.


\bibitem{aHKR2} A. an Huef, S. Kang, I. Raeburn, {\em Spatial realisations of KMS states on the $C^{*}$-algebras of higher-rank graphs}, J. Math. Anal. Appl. {\bf 427} (2015), 977-1003. https://doi.org/10.1016/j.jmaa.2015.02.045.


\bibitem{aHLRS1}  A. an Huef, M. Laca, I. Raeburn, A. Sims, {\em KMS states on the $C^{*}$-algebras of reducible graphs}, Ergodic Theory and Dynamical Systems  {\bf 35(8)} (2015), 2535-2558.  http://dx.doi.org/10.1017/etds.2014.52.

\bibitem{aHLRS} A. an Huef, M. Laca, I. Raeburn, A. Sims, {\em KMS states on the $C^{*}$-algebra of a higher-rank graph and periodicity in the path space}, J. Func. Analysis {\bf 268} (2015), 1840-1875. https://doi.org/10.1016/j.jfa.2014.12.006.




\bibitem{KP} A. Kumjian, D. Pask, {\em Higher Rank Graph $C^{*}$-Algebras}, New York J. Math. {\bf 6}, (2000), 1-20. 
 

\bibitem{N} S. Neshveyev, {\em KMS states on the $C^*$-algebras of
    non-principal groupoids}, J. Operator Theory {\bf 70 (2)}  (2013) , 513-530. http://dx.doi.org/10.7900/jot.2011sep20.1915 .




\bibitem{Re} J. Renault, {\em A Groupoid Approach to $C^*$-algebras},  LNM 793, Springer Verlag, Berlin, Heidelberg, New York, 1980.  

\bibitem{RS} D. Robertson, A. Sims, {\em Simplicity of $C^{*}$-algebras associated to higher-rank graphs}, Bulletin of the London Mathematical Society {\bf 39 (2)} (2007), 337-344. https://doi.org/10.1112/blms/bdm006.

\bibitem{Ru} W. Rudin, {\em Fourier Analysis on Groups}, Interscience Tracts in Pure and Applied Mathematics number 12, Interscience Publishers, New York, 1962.






\bibitem{Th2} K. Thomsen, {\em $KMS$-states and conformal measures}, Comm. Math. Phys. {\bf 316} (2012), 615-640. https://doi.org/10.1007/s00220-012-1591-z.


\bibitem{Yeend} T. Yeend, { \em Groupoid  models for the $C^{*}$-algebras of topological higher-rank graphs}, Journal of Operator Theory {\bf 57 (1)} (2007), 95-120. http://www.jstor.org/stable/24715760.

\end{thebibliography}
\end{document}